\documentclass[11pt,reqno]{amsart}
\usepackage[usenames]{color}
\usepackage{amsmath, amssymb, bm, float, latexsym, multicol, amsthm, color, enumitem, tensor, comment, mathrsfs, tikz}
\usepackage[all]{xy}
\usepackage{cancel}
\usepackage[T1]{fontenc}
\usepackage{hyperref}
\usepackage{cleveref}
\newtheoremstyle{remark}
{\parsep}   
{\parsep}   
{\normalfont}  
{0pt}       
{\bfseries} 
{.}         
{5pt plus 1pt minus 1pt} 
{}          
\theoremstyle{plain}
\newtheorem{thm}{Theorem}[section]

\newtheorem{lem}[thm]{Lemma}
\newtheorem{prp}[thm]{Proposition}

\theoremstyle{remark}

\newtheorem*{rmk}{Remark}

\numberwithin{equation}{section}

\crefname{thm}{theorem}{theorems}
\crefname{lem}{lemma}{lemmata}
\crefname{prp}{proposition}{propositions}
\crefname{cor}{corollary}{corollaries}
\crefname{figure}{figure}{figures}

\newcommand{\C}{\mathbb{C}}
\newcommand{\N}{\mathbb{N}}

\newcommand{\R}{\mathbb{R}}
\newcommand{\Z}{\mathbb{Z}}

\newcommand{\mc}{\mathcal}
\newcommand{\mf}{\mathfrak}

\newcommand{\cb}[1]{\left\{{#1}\right\}}
\newcommand{\cbm}[2]{\left\{{#1}\; \middle|\; {#2}\right\}}

\newcommand{\rb}[1]{\left({#1}\right)}
\newcommand{\Brb}[1]{\Big({#1}\Big)} 
\newcommand{\vb}[1]{\left| {#1} \right|}
\newcommand{\Bvb}[1]{\Big| {#1} \Big|} 
\newcommand{\bs}{\backslash}
\newcommand{\erfc}{\operatorname{erfc}}
\newcommand{\fa}{\;\forall}

\newcommand{\Li}{\operatorname{Li}}
\newcommand{\pb}{\operatorname{pb}}
\newcommand{\pd}{\operatorname{pd}}
\newcommand{\ppmod}[1]{\hspace{-0.2cm}\pmod{#1}}
\newcommand{\re}{\operatorname{Re}}
\newcommand{\sbe}{\subseteq}

\title{Distributions of parity differences and biases in partitions into distinct parts}

\author{Siu Hang Man}
\address{Charles University, Faculty of Mathematics and Physics, Department of Algebra, Sokolovská 83, 186~75 Praha~8, Czechia}
\email{shman@karlin.mff.cuni.cz}
\thanks{This project has received funding from the European Research Council (ERC) under the European Union’s Horizon 2020 research and innovation programme (grant agreement No. 101001179). The project is also supported by the Czech Science Foundation GAČR grant 21-00420M, Charles University programme PRIMUS/24/SCI/010 and PRIMUS/25/SCI/008, as well as the OP RDE project No. CZ.02.2.69/0.0/0.0/18\_053/0016976 International mobility of research, technical and administrative staff at the Charles University.}

\begin{document}

\begin{abstract}
For a partition $\lambda \vdash n$, we let $\pd(\lambda)$, the parity difference of $\lambda$, be the number of odd parts of $\lambda$ minus the number of even parts of $\lambda$. We prove for $c_0\in\R$ an asymptotic expansion for the number of partitions of $n$ into distinct parts with normalised parity difference $n^{- 1/4}\pd(\lambda)$ greater than $c_0$ as $n\to \infty$. As a corollary, we find the distribution of the parity differences and parity biases for partitions of $n$ into distinct parts. We also establish analogous results for generalised parity differences modulo $N$.
\end{abstract}
\date{}
\makeatletter
\@namedef{subjclassname@2020}{%
  \textup{2020} Mathematics Subject Classification}
\makeatother
\maketitle

\section{Introduction}

A \emph{partition} of a non-negative integer $n$ is a non-increasing sequence of positive integers whose sum is $n$. Given a partition $\lambda \vdash n$, we denote by $\pd(\lambda)$ the \emph{parity difference} of $\lambda$, defined as the number of odd parts of $\lambda$ minus the number of even parts of $\lambda$. For example, for $\lambda = (4,3,1) \vdash 8$ we have $\pd(\lambda) = 2-1 = 1$. 

Back in 2007, Dartyge and Szalay \cite{DS07DR} proved that for $n\in\N$ sufficiently large, there are more (unrestricted) partitions of $n$ with positive parity differences than partitions of $n$ with negative parity differences. This phenomenon is now known as \emph{parity bias} in partitions. Recently, Kim, Kim, and Lovejoy \cite{KKL20PB} employed also combinatorial arguments, and showed that the statement holds true whenever $n\ge 3$. It was then shown in \cite{BBDMS22PB,BMRS24AP} that parity bias (or reverse parity bias, i.e. having more partitions with negative parity differences) occurs in a wide family of partitions, such as partitions into distinct parts, and other kinds of restricted partitions. One could also compare the number of parts of a partition which lie in residue classes modulo $N$; parity bias is the case where $N=2$. Such results are found in \cite{BMRS24AP,Che22FR,DS07DR,KK21BI}. 

While there are parity biases for both unrestricted partitions and partitions into distinct parts, they have rather different asymptotic behaviours. Let $\mc P(n)$ (resp. $\mc D(n)$) denote the set of partitions (resp. partitions into distinct parts) of $n$. Denote
\begin{align*}
p_o(n) &:= \#\cbm{\lambda \in \mc P(n)}{\pd(\lambda) > 0}, & p_e(n) &:= \#\cbm{\lambda \in \mc P(n)}{\pd(\lambda) < 0},\\
d_o(n) &:= \#\cbm{\lambda \in \mc D(n)}{\pd(\lambda) > 0}, & d_e(n) &:= \#\cbm{\lambda \in \mc D(n)}{\pd(\lambda) < 0}.\\
\end{align*}
It was first shown in \cite{DS07DR,DS11LD,DS12LD} and later reproved in \cite{BMRS24AP,KKL20PB} that
\begin{align*}
\lim_{n\to\infty} \frac{p_o(n)}{p_e(n)} &= 1+\sqrt{2}, & \lim_{n\to\infty} \frac{d_o(n)}{d_e(n)} &= 1.
\end{align*}
While it is known from \cite{BBDMS22PB} that $d_o(n)>d_e(n)$ for all $n\ge 20$, the asymptotic formula above says that $d_e(n)$ is almost as large as $d_o(n)$. On the other hand, the asymptotic formula by Dartyge, Sárközy, and Szalay \cite{DSS06DS} shows that almost all partitions into distinct parts have approximately the same number of odd and even parts. This means that the parity difference $\pd(\lambda)$ is usually small. The goal of the paper is to study the asymptotic distribution of the parity differences $\pd(\lambda)$ for $\lambda \in \mc D(n)$ as $n\to \infty$, as well as the distribution of the parity biases.

Let $N\in\N_{\ge 2}$, $1\le \alpha, \beta \le N$, and $\alpha\ne\beta$. For a partition $\lambda = (\lambda_1,\ldots,\lambda_l) \vdash n$, we define
\[
\pd_{\alpha,\beta;N}(\lambda) := \#\cb{\lambda_i \equiv \alpha\ppmod{N}} - \#\cb{\lambda_i \equiv \beta\ppmod{N}}
\]
the \emph{parity difference} of $\lambda$ with respect to the residue classes $\alpha,\beta\pmod{N}$. The standard parity difference $\pd(\lambda)$ introduced above is then identified as $\pd(\lambda) = \pd_{1,2;2}(\lambda)$. To count the number of partitions with given parity difference, we introduce the following notations. Let $c\in\R$ be a parameter. We write 
\[
\mc D_{\alpha,\beta;N;c}(n) := \cbm{\lambda \in \mc D(n)}{\pd_{\alpha,\beta;N}(\lambda) \ge c}
\]
the set of partitions with parity differences at least $c$ with respect to the residue classes $\alpha,\beta\pmod{N}$, and $d_{\alpha,\beta;N;c}(n) := \# \mc D_{\alpha,\beta;N;c}(n)$. For our purpose, we shall set $c = c_0 n^{1/4}$, where $c_0 \in \R$ is a fixed constant. Our first main theorem gives an asymptotic formula for $d_{\alpha,\beta;N;c_0 n^{1/4}}(n)$ as $n\to \infty$. To state the theorem, we need some notations. For $\ell\in\Z$, we let $[\ell]_N$ denote the smallest non-negative integer congruent to $\ell\pmod{N}$. For fixed $c_0\in\R$, we write $\partial = \partial(n) := \lceil c_0n^{1/4}\rceil - c_0n^{1/4}$.

\begin{thm}\label{thm:main1}
Let $N\in\N_{\ge 2}$, $1\le\alpha,\beta\le N$, $\alpha\ne\beta$, and $c_0\in\R$ a fixed constant. Then we have as $n\to\infty$
\begin{multline}\label{eq:main1}
d_{\alpha,\beta;N;c_0n^{1/4}}(n) = e^{\pi\sqrt{n/3}} n^{-3/4} \sum_{\substack{\bm\ell\in(\Z/N\Z)^N\\ NH(\bm\ell) \equiv n \ppmod{N}}} \left(\frac{1}{8\sqrt[4]{3}N^{N-1}} \erfc\rb{\frac{c_0\sqrt{\pi N}}{2\sqrt[4]{3}}}\right.\\
\left.{}+ \frac{e^{-c_0^2\pi N/4\sqrt{3}} (N^2-2\partial^*N+(\beta-\alpha))}{16\sqrt{3} N^{N-1/2} n^{1/4}} + O\rb{n^{-1/2}}\right).
\end{multline}
Here $\erfc(x)$ denotes the complementary error function, $H$ is a quadratic polynomial given in \eqref{eq:H_def}, and $\partial^*$ is given by $\partial^* = \partial^*(n,\bm\ell) := \partial + [\ell_\alpha - \ell_\beta - \lceil c_0 n^{1/4} \rceil]_N$.
\end{thm}

From the expression \eqref{eq:main1}, we see that the main term is not sensitive to the chosen residue classes $\alpha$ and $\beta$. However, in the second term of the asymptotic formula we see the dependence of $\alpha$ and $\beta$. In particular, the functions $d_{\alpha,\beta;N;c_0n^{1/4}}(n)$ and $d_{\beta,\alpha;N;c_0n^{1/4}}$ have different asymptotic behaviours, illustrating the parity bias phenomenon. Beside the parity bias, we also see a contribution from $\partial$ (or $\partial^*$), which signifies the discontinuity when the threshold $c_0 n^{1/4}$ grows over an integer. In general, the asymptotic behaviour also depends on the congruence class of $n\pmod N$.

If $N=2$ or $N\ge 5$, then the asymptotic formula \eqref{eq:main1} can be significantly simplified. In particular, in this case the asymptotic formula does not depend on the congruence classes of $n$ and $\lceil c_0 n^{1/4} \rceil \pmod{N}$.

\begin{thm}\label{thm:main1a}
Suppose $N=2$ or $N\ge 5$. Then we have as $n\to\infty$
\begin{multline*}
d_{\alpha,\beta;N;c_0n^{1/4}}(n) = e^{\pi\sqrt{n/3}} n^{-3/4}\\
\times \rb{\frac{1}{8\sqrt[4]{3}} \erfc\rb{\frac{c_0\sqrt{\pi N}}{2\sqrt[4]{3}}} + \frac{e^{-c_0^2\pi N/4\sqrt{3}}}{16\sqrt{3N}} \rb{\beta-\alpha+N-2N\partial} n^{-1/4} + O\rb{n^{-1/2}}}.
\end{multline*}
\end{thm}

Using the main term of the asymptotic formula \eqref{eq:main1}, we obtain the distribution of the parity differences $\pd_{\alpha,\beta;N}(\lambda)$ for $\lambda \in \mc D(n)$, as $n\to \infty$. This recovers a result of Dartyge--Szalay \cite{DS11LD,DS12LD}.
\begin{thm}\label{thm:main2}
Let $N\in\N_{\ge 2}$, $1\le \alpha, \beta \le N$, $\alpha\ne\beta$. Then, as $n\to\infty$, the normalised parity differences $n^{-1/4} \pd_{\alpha,\beta;N}(\lambda)$ for $\lambda\in\mc D(n)$ are normally distributed with mean $0$ and variance $\frac{2\sqrt{3}}{\pi N}$. In other words, for any real numbers $a\le b$, we have
\[
\lim_{n\to\infty} \frac{\#\cbm{\lambda\in\mc D(n)}{a \le n^{-1/4}\pd_{\alpha,\beta;N}(\lambda) \le b}}{\#\mc D(n)} = \frac{\sqrt{N}}{2\sqrt[4]{3}}\int_a^b e^{-\frac{\pi Nx^2}{4\sqrt{3}}} dx.
\]
\end{thm}

When we consider also the second term of the asymptotic formula, we obtain also the distribution of parity biases. Before we state the theorem, we make precise what this means. For $n\in\N$, $1\le \alpha,\beta\le N$, $\alpha\ne\beta$, and $c\in\N_0$, we define the \emph{parity bias} of the set $\mc D(n)$ with bias $c$ with respect to the residue classes $\alpha,\beta \pmod{N}$ to be the difference
\[
\pb_{\alpha,\beta;N}(\mc D(n),c) := \#\cbm{\lambda \in \mc D(n)}{\pd_{\alpha,\beta;N}(\lambda) = c} - \#\cbm{\lambda \in \mc D(n)}{\pd_{\beta,\alpha;N}(\lambda) = c}.
\]
The following theorem gives the distribution of the parity biases.
\begin{thm}\label{thm:main2a}
Let $N=2$ or $N\ge 5$, $1\le \alpha,\beta\le N$, $\alpha\ne\beta$. Then, for any real numbers $0\le a\le b$, we have
\[
\lim_{n\to\infty} \frac{\sum_{a\le n^{-1/4}c\le b} \pb_{\alpha,\beta;N}(\mc D(n), c)}{d_{\alpha,\beta;N;0} - d_{\beta,\alpha;N;0}} = \frac{\pi N}{2\sqrt{3}} \int_a^b x e^{-\frac{\pi N x^2}{4\sqrt{3}}} dx.
\]
\end{thm}

In particular, this says the parity bias is the greatest when the bias is around $\sqrt[4]{12}/\sqrt{\pi N}$. 

\begin{rmk}
For simplicity, we only stated the distribution of parity biases for $N=2$ and $N\ge 5$. For $N=3,4$, because the asymptotic behaviour of $d_{\alpha,\beta;N,c_0n^{1/4}}(n)$ depends also on the congruence classes of $n$ and $\lceil c_0n^{1/4}\rceil \pmod{N}$, we can have such a distribution result only if we restrict $n$ and $\lceil c_0n^{1/4}\rceil$ to particular congruence classes, or consider an average over multiple values of $n$. 
\end{rmk}

The main tool for the proof of \Cref{thm:main1} is the circle method, more specifically the saddle point method. A classical application of the circle method sees that we write the generating function for $d_{\alpha,\beta;N;c}(n)$ as a $q$-series (see \eqref{eq:generating_function}), which is known already in \cite{DS11LD,DS12LD}, and derive an asymptotic expansion for the generating function as $q$ approaches a root of unity, as holomorphic functions in $z$ (where $q = e^{-z}$). 

Unfortunately, the asymptotic expansion constructed this way only holds for a narrow cone around the positive real line and is insufficient for the computation of the major arc contribution. As a remedy, we use an alternative asymptotic expansion (see \Cref{prp:wide_asymp}) from \cite{BMRS24AP}, which is no longer holomorphic in $z$, but holds in a sufficiently large region so that the circle method can be applied. The main novelty of this article compared to the approach in \cite{DS11LD,DS12LD} is the use of Euler--Maclaurin summation formula, which allows the extraction of not just the main term but also other terms of smaller magnitude, revealing the parity biases, whose sizes are shadowed by the main term. With a careful analysis of all the integrals that arise, we are able to show in \Cref{prp:d_asymp} that $d_{\alpha,\beta;N,c_0n^{1/4}}(n)$ admits an asymptotic expansion in $n^{-1/4}$, whose coefficients are completely explicit and depend only on $\alpha,\beta,c_0,N,\partial$, as well as $n,\lceil c_0n^{1/4}\rceil \pmod{N}$. 

The paper is organised as follows. In \Cref{section:prelim}, we recollect the asymptotic formulae that we use in later sections, mainly from the toolbox in \cite{BMRS24AP}. In \Cref{section:asymp}, we compute the asymptotic expansions arising from the circle method. In \Cref{section:construction} we find an explicit characterisation of the asymptotic expansion for $d_{\alpha,\beta;N;c_0n^{1/4}}(n)$. Finally, in \Cref{section:coefficients} we evaluate explicitly the coefficients of the asymptotic expansions and prove the theorems. We end the paper with some numerical illustrations of the theorems.

\section*{Acknowledgements}

We would like to thank the anonymous referee for valuable comments and helpful suggestions.

\section{Preliminaries}\label{section:prelim}

\subsection{Generating functions for partitions}
The generating function for $d(n):= \#\mc D(n)$, the number of partitions of $n$ into distinct parts, is given by
\[
\sum_{n=0}^\infty d(n) q^n = \sum_{n\ge 0}^\infty \frac{q^{\frac{n(n+1)}{2}}}{(q;q)_n},
\]
where $(a;q)_n := \prod_{j=0}^{n-1} (1-aq^j)$ is the usual $q$-Pochhammer symbol. It is a classical result \cite{Hua42NP} that we have an asymptotic formula
\begin{equation}\label{eq:Hua}
d(n) = \frac{e^{\pi\sqrt{n/3}}}{4\sqrt[4]{3} n^{3/4}} \rb{1+O(n^{-1/2})}.
\end{equation}
For our purpose, we introduce some other generating functions related to partitions. Let $\bm m = (m_1,\ldots, m_N) \in \N_0^N$. Using standard techniques in combinatorics, the generating function for the number of partitions $d_{\bm m}(n)$ of $n$ into $m_j$ distinct parts congruent to $j\pmod{N}$, $1\le j \le N$, is given by (cf. \cite{DS11LD,BMRS24AP})
\begin{equation}\label{eq:generating_summand}
\sum_{n=0}^\infty d_{\bm m}(n)q^n = \frac{q^{NH(\bm m)}}{\prod_{j=1}^N (q^N;q^N)_{m_j}},
\end{equation}
where
\begin{align}\label{eq:H_def}
H(\bm m) &:= \tfrac 12 \bm m^T\bm m + \bm b^T \bm m, & \bm b := \begin{pmatrix}\frac 1N-\frac 12, & \frac 2N -\frac 12, & \ldots, & \frac 12\end{pmatrix}^T.
\end{align}
The generating function for $d_{\alpha,\beta;N;c}(n)$ is then obtained simply by summing over the generating functions \eqref{eq:generating_summand} with $m_\alpha-m_\beta \ge c$:
\begin{equation}\label{eq:generating_function}
D_{\alpha,\beta;N;c}(q) := \sum_{n\ge 0} d_{\alpha,\beta;N;c}(n) q^n = \sum_{\bm m \in \mc S_{\alpha,\beta;N;c}} \frac{q^{NH(\bm m)}}{\prod_{j=1}^N (q^N;q^N)_{m_j}},
\end{equation}
where $\mc S_{\alpha,\beta;N;c} := \cbm{(m_1,\ldots,m_N)\in\N_0^N}{m_\alpha-m_\beta\ge c}$.

\subsection{Special functions}
We first recall several special functions. For $s\in\C$, the {\it polylogarithm} is given by
\[
\Li_s(w) := \sum\limits_{n\geq1} \frac{w^n}{n^s}, \quad \vb{w} < 1.
\]
For $w\in(0,1)$, define the {\it Rogers dilogarithm function} (shifted by a constant so that $L(1)=0$)
\[
L(w) := \Li_2(w) + \frac12\log(w)\log(1-w) - \frac{\pi^2}{6}.
\]
It is well-known (see \cite[p.6]{Zag07DF}) that
\[
L\rb{\frac12} = -\frac{\pi^2}{12}.
\]
Let $B_r(x)$ denote the \textit{$r$-th Bernoulli polynomial} defined via the generating function
\[
\frac{t e^{xt}}{e^t-1} =: \sum\limits_{n=0}^\infty B_n(x) \frac{t^n}{n!}.
\]
The \textit{$r$-th Bernoulli number} is given by $B_r := B_r(0)$.

\subsection{Asymptotic formulae}

For the proof we make use of the following version of the Euler--Maclaurin summation formula. The classical version of Euler--Maclaurin summation compares the sum $\sum_{n=0}^mF(n)$ to the integral $\int_0^mF(x)dx$; this version makes use of the same expressions but with a change of variables $F(x)=f(xz+a)$. A related ``shifted'' version of the Euler--Maclaurin summation formula can also be found in \cite[Proposition 3]{Zag06MT}.

\begin{prp}\label{prp:EMF}
Let $a,z\in\C$. Let $f$ be a holomorphic function that is defined on the ray $L(a,z):=\{tz+a:t\in\R_{\ge 0}\}$. Then we have, for all $m,R\in\N$,
\begin{multline*}
\sum_{n=0}^m f(nz+a) = \frac1z\int_a^{a+mz} f(x) dx + \frac{f(mz+a)+f(a)}{2}\\
{}+ \sum_{r=1}^R \frac{B_{2r}z^{2r-1}}{(2r)!}\rb{f^{(2r-1)}(mz+a)-f^{(2r-1)}(a)} + O(1)z^{2R}\int_a^{a+mz} \vb{f^{(2R+1)}(x)} dx.
\end{multline*}
Furthermore, if $f$ and all of its derivatives have rapid decay (i.e. decay faster than any power of $z$) on $L(a,z)$, then we have
\begin{align*}
\sum_{n\ge0} f(nz+a) &= \frac1z\int_a^{a+\infty z} f(x) dx + \frac{f(a)}{2}\\
&\hspace{2cm}- \sum_{r=1}^R \frac{B_{2r}z^{2r-1}}{(2r)!}f^{(2r-1)}(a) + O(1)z^{2R}\int_a^{a+z\infty} \vb{f^{(2R+1)}(x)} dx.
\end{align*}
\end{prp}
\begin{rmk}
We give a note on the usage of the Euler--Maclaurin summation formula in the article. Depending on the expression that we sum over, the size of the error term may decay at different rates as $R\to\infty$. Hence, at different instances we may need to apply \Cref{prp:EMF} with different choices of $R$, in order to achieve a given level of precision in the end results. The canonical usage of the formula is to convert a sum into an integral, but at a few places we also apply the formula backwards, and convert an integral into a sum.
\end{rmk}

Here we recall some results in \cite{BMRS24AP} about the asymptotic formulae for the function
\begin{equation}\label{eq:summand_def}
\frac{e^{-H(\bm m)z}}{\prod_{j=1}^N (e^{-z};e^{-z})_{m_j}}
\end{equation}
as $z\to 0$ in the right half-plane. The readers are referred to \cite{BMRS24AP} for the full proofs of the statements in this subsection. We shall use $\ll$, $\gg$, and $\asymp$ to compare the size of (possibly complex) quantities, without any assumption on their signs or arguments unless stated otherwise. For example, for $y\in\R$, $y\ll 1$ means $-C\le y \le C$ for some $C>0$.

Let $\varepsilon > 0$, and $z=\varepsilon(1+iy) \in \C$. For $\lambda\in\R$, we define
\begin{equation}\label{eq:N_def}
\mc N_{\varepsilon,\lambda} := \cbm{\bm m \in \N_0^N}{\bm m = \frac{\log(2)}{\varepsilon}\bm 1 + \bm\mu, \; \vb{\bm\mu}\le \varepsilon^\lambda},
\end{equation}
where $\vb{\bm\mu}$ denotes the (Euclidean) norm of the vector $\bm\mu\in\R^N$. Throughout, we always assume that $-\frac 23 < \lambda < -\frac 12$. As in \cite{BMRS24AP}, we give two different asymptotic formulae for \eqref{eq:summand_def}, which apply for different ranges of $y$.

\begin{rmk}
When $\varepsilon = \re(z)$ is fixed, the expression \eqref{eq:summand_def} has the greatest size when $\bm m$ is around $\frac{\log(2)}{\varepsilon} \bm 1$. In particular, it is shown in \cite{VZ11NC} that whenever $\lambda < -\frac 12$, the summands \eqref{eq:summand_def} for $\bm m \not\in \mc N_{\varepsilon,\lambda}$ have essentially negligible contribution to the sum \eqref{eq:generating_function} (see \Cref{prp:out_of_range}). This allows us to truncate the sum \eqref{eq:generating_function}, which helps with the analysis. On the other hand, for the asymptotic analysis in \Cref{prp:narrow_asymp,prp:wide_asymp} to work, we need $\lambda > -\frac 23$. This motivates the choice of $\lambda$ above. We note that the possibility to choose such $\lambda$ plays a pivotal role in various literature on asymptotic estimates of Nahm sums in the literature \cite{BMRS24AP,GZ21AN,VZ11NC,Zag07DF}.
\end{rmk}

\begin{prp}[{\cite[Proposition 3.2]{BMRS24AP}}]\label{prp:narrow_asymp}
Let $\bm m\in\mc N_{\varepsilon,\lambda}$ and $\bm u = (u_1,\ldots, u_N)\in\C^N$ be such that, for $1\le j \le N$,
\begin{equation}\label{eq:uj_def}
m_j = \frac{\log (2)}{z} + \frac{u_j}{\sqrt{z}}.
\end{equation}
If $y\ll\varepsilon^{1/3+\delta}$ for some $\delta>0$, then as $z\to 0$ we have for $R\in\N$, uniformly in $\bm u$,
\begin{equation*}
\frac{e^{-H(\bm m)z}}{\prod_{j=1}^N (e^{-z};e^{-z})_{m_j}} = \rb{\frac z\pi}^{N/2} \frac{e^{\frac{\pi^2N}{12z} - \bm u^T\bm u}}{\sqrt{2}} \rb{\sum_{r=0}^{R-1} C_r(\bm u) z^{r/2} + O\rb{\varepsilon^{3R\delta_1}}},
\end{equation*}
where $\delta_1:=\min\{\delta,\frac23+\lambda\}>0$. The polynomials $C_r(\bm u)$ are defined as coefficients of the formal exponential
\begin{equation}\label{eq:C_def}
\sum_{r\ge0} C_r(\bm u)z^{r/2} := \exp(\phi(\bm u,z)),\qquad \phi(\bm u,z) := -\bm b^T\bm u\sqrt{z} - \frac{Nz}{24} + \sum_{j=1}^N \xi\rb{\frac{u_j}{\sqrt{z}},z},
\end{equation}
and the function $\xi(\nu,z)$ has an asymptotic expansion as $z\to 0$ given by
\begin{equation*}
\xi(\nu,z) = -\sum_{r=2}^{R-1} \rb{B_r(-\nu)-\delta_{r,2}\nu^2}\Li_{2-r}\rb{\frac12}\frac{z^{r-1}}{r!} + O\rb{z^{R-1}}.
\end{equation*}
In particular, we have $\deg C_r(\bm u) \le 3r$.
\end{prp}

Let $\mc A:=(1+\frac{2^{-(1+iy)}}{1-2^{-(1+iy)}})I_N$, with $I_N$ the $N\times N$ identity matrix, and define
\begin{equation}\label{eq:Lambda_def}
	\Lambda(y) := N\rb{\frac{\pi^2}{6}-\frac{\log(2)^2(1+iy)^2}{2}-\Li_2\rb{2^{-(1+iy)}}}.
\end{equation}

\begin{prp}[{\cite[Proposition 3.4]{BMRS24AP}}]\label{prp:wide_asymp}
Let $\bm m\in\mc N_{\varepsilon,\lambda}$ and $\bm v = (v_1,\ldots,v_N) \in\C^N$ be such that, for $1\le j \le N$, 
\begin{equation*}
m_j = \frac{\log (2)}{\varepsilon} + \frac{v_j}{\sqrt{z}}.
\end{equation*}
Suppose $y \ll\varepsilon^{-1-3\lambda/2+\delta}$ for some $\delta>0$. Then as $z\to 0$ we have for $R\in\N$, uniformly in $\bm v$,
\begin{multline*}
\frac{e^{-H(\bm m)z}}{\prod_{j=1}^N (e^{-z};e^{-z})_{m_j}}  = 2^{-(1+iy)/2} \rb{1-2^{-(1+iy)}}^{-N/2}\rb{\frac{z}{2\pi}}^{N/2} e^{\Lambda(y)/z}\\
\times e^{-\frac 12\bm v^T\mc A\bm v+\frac{1}{\sqrt z}\rb{-\log(2)(1+iy)-\log(1-2^{-(1+iy))}}\sum_{j=1}^N v_j}\rb{\sum_{r=0}^{R-1} D_r(\bm v,y)z^{r/2} + O\rb{\varepsilon^{2R\delta_2}}},
\end{multline*}
where $\delta_2:=\min\{1+\frac{3\lambda}{2},\delta\}>0$. The functions $D_r(\bm v,y)$ are defined as coefficients of the formal exponential
\begin{equation*}
\sum_{r\geq0} D_r(\bm v,y)z^{r/2} := \exp(\phi(\bm v,z)),\qquad \phi(\bm v,z) := -\bm b^T\bm v\sqrt{z} - \frac{Nz}{24} + \sum_{j=1}^N \xi_y\rb{\frac{v_j}{\sqrt{z}},z},
\end{equation*}
and $\xi_y(\frac{v_j}{\sqrt{z}},z)$ has an asymptotic expansion as $z\to 0$ given by
\begin{equation*}
\xi_y\rb{\nu,z} = - \sum_{r=2}^{R-1} \rb{B_r\rb{-\nu} - \delta_{r,2} \nu^2} \Li_{2-r}\rb{2^{-(1+iy)}} \frac{z^{r-1}}{r!} + O\rb{z^{R-1}}.
\end{equation*}
\end{prp}
\begin{rmk}
Under the assumptions in \Cref{prp:wide_asymp}, we have $z\ll \varepsilon^{-3\lambda/2+\delta_2}$, and $v_j/\sqrt{z} \le \varepsilon^\lambda$. These facts will be used below in the proofs.
\end{rmk}

We also need the error estimates from \cite{BMRS24AP}.

\begin{prp}[{\cite[Proposition 4.1]{BMRS24AP}}]\label{prp:out_of_range}
For all $L\in\N$, as $\re(z)\to 0$ in the right half-plane, we have
\[
\sum_{\bm m \in \N_0^N \bs \mc N_{\varepsilon,\lambda}} \vb{\frac{e^{-H(\bm m)z}}{\prod_{j=1}^N (e^{-z};e^{-z})_{m_j}}} \ll \varepsilon^L e^{\frac{\pi^2N}{12\varepsilon}}.
\]
\end{prp}

\begin{prp}[{\cite[Proposition 4.2]{BMRS24AP}}]\label{prp:minor_arc}
Let $\bm m \in \mc N_{\varepsilon,\lambda}$, and suppose that $\varepsilon^{-\delta} < \vb{y} \le \frac{\pi}{\varepsilon}$ for some $\delta > 0$. Then we have, for all $L\in\N$,
\[
\frac{e^{-H(\bm m)z}}{\prod_{j=1}^N (e^{-z};e^{-z})_{m_j}} \ll \varepsilon^L e^{\frac{\pi^2N}{12\varepsilon}}. 
\]
\end{prp}

\section{Asymptotic estimates}\label{section:asymp}

In this section, we establish some asymptotic behaviours for the function $D_{\alpha,\beta;N;c}(q)$. Throughout the section, the parameters $\alpha,\beta,N,c$ are fixed, and are suppressed from the notations when no confusion arises.

First we split $D(q) = D_{\alpha,\beta;N;c}(q)$ as a sum
\begin{equation}\label{eq:D_def}
D(q) = \sum_{\bm\ell \in (\Z/N\Z)^N} D_{\bm\ell}(q), \quad \text{ where } \quad D_{\bm\ell}(q) = D_{\alpha,\beta;N;c;\bm\ell}(q) := \sum_{\bm m\in\mc S_{\bm\ell}}\frac{q^{NH(\bm m)}}{\prod_{j=1}^N (q^N;q^N)_{m_j}},
\end{equation}
and $\mc S_{\bm\ell} = \mc S_{\alpha,\beta;N;c;\bm\ell} := \cbm{\bm m\in\mc S_{\alpha,\beta;N;c}}{\bm m \equiv \bm\ell \pmod{N}}$. We always pick the representative for $\bm\ell = (\ell_1,\ldots,\ell_N) \in (\Z/N\Z)^N$ with $\ell_j \in \cb{0,1,\ldots,N-1}$. To apply the circle method, we need a change of variables $q \mapsto e^{-z}$. Write $\zeta_N := e^{2\pi i/N}$, and define
\[
g_{\bm\ell}(z) = g_{\alpha,\beta;N;c;\bm\ell}(z) := \sum_{\bm m \in \mc S_{\alpha,\beta;N;c;\bm\ell}} \frac{e^{-H(\bm m)z}}{\prod_{j=1}^N(e^{-z};e^{-z})_{m_j}}.
\]
Then for $k\in\Z$ we have
\[
D_{\bm\ell} (\zeta_N^k e^{-z}) = \zeta_N^{NH(\bm\ell)k} g_{\bm\ell}(Nz).
\]
We split
\begin{equation}\label{eq:g_split}
g_{\bm\ell}(z) = g_{\bm\ell}^{[1]}(z) + g_{\bm\ell}^{[2]}(z),
\end{equation}
where
\begin{align*}
g_{\bm\ell}^{[1]}(z) = g_{\alpha,\beta;N;c;\bm\ell}^{[1]}(z) &= \sum_{\bm m \in \mc S_{\bm\ell} \cap \mc N_{\varepsilon,\lambda}} \frac{e^{-H(\bm m)z}}{\prod_{j=1}^N (e^{-z};e^{-z})_{m_j}},\\
g_{\bm\ell}^{[2]}(z) = g_{\alpha,\beta;N;c;\bm\ell}^{[2]}(z) &= \sum_{\bm m \in \mc S_{\bm\ell} \bs \mc N_{\varepsilon,\lambda}} \frac{e^{-H(\bm m)z}}{\prod_{j=1}^N (e^{-z};e^{-z})_{m_j}}.
\end{align*}

From \Cref{prp:out_of_range}, we see that the contribution from the function $g_{\bm\ell}^{[2]}$ is negligible. The rest of the section provides asymptotic estimates of the function $g_{\bm\ell}^{[1]}$. For convenience, we set for $r\in\N_0$
\[
P_r(\bm u) := C_r(\bm u) e^{-\bm u^T\bm u},
\]
where the polynomial $C_r(\bm u)$ is defined in \eqref{eq:C_def}; in particular, we have $\deg(C_r)\le3r$. For the statements below,  the $\alpha$-th component $u_\alpha$ of $\bm u$ plays a special role. To emphasise this component, we write $\bm u = (\bm{u_{[1]}}, u_\alpha)$, where $\bm{u_{[1]}}$ denotes the remaining $N-1$ components of $\bm u$. Note that this notation does not say that $u_\alpha$ is the $N$-th component of $\bm u$. We write $\kappa = \kappa_{\alpha,\beta;N;c,\bm\ell}$ to denote the smallest integer such that $\kappa \ge c$ and $\kappa \equiv \ell_\alpha-\ell_\beta \pmod{N}$. We also recall that $\lambda$ is a fixed real number satisfying $-\frac 23 < \lambda < -\frac 12$ (see the discussion below \eqref{eq:N_def}).

\begin{prp}\label{prp:g1_asymp}
Suppose $y \ll \varepsilon^{1+\lambda+\delta}$ for some $\delta>0$. Then as $z\to 0$ we have for $R\in\N$
\begin{multline*}
g_{\bm\ell}^{[1]}(z) = \frac{e^{\frac{\pi^2N}{12z}}}{\sqrt{2}\pi^{N/2} N^N} \sum_{r=0}^{R-1} z^{r/2} \int_{\R^{N-1}} \int_{u_\beta+\kappa\sqrt{z}}^\infty P_r(\bm{u_{[1]}},u_\alpha) du_\alpha d\bm{u_{[1]}}\\
+ e^{\frac{\pi^2N}{12z}-\frac{\kappa^2z}2} \sum_{r=0}^{S(R)-1} W_r(\kappa) z^{(r+1)/2} + O\Brb{e^{\frac{\pi^2N}{12\varepsilon}}\varepsilon^{N\rb{\lambda+1/2}+3R\rb{\lambda+2/3}}},
\end{multline*}
where $S(R) := 4(R-1) + 2\big\lceil\frac{1/2 - \rb{\lambda+1/2}(3R+N-2)}{2\lambda+2}\big\rceil$, and $W_r$ are polynomials of degree $\le r$. In particular, for every $L\in\N$, we have as $z\to 0$, with $R_0 = R_0(L) := \big\lceil \frac{L-N(\lambda+1/2)}{3\lambda+2}\big\rceil$
\begin{multline*}
g_{\bm\ell}^{[1]}(z) = \frac{e^{\frac{\pi^2N}{12z}}}{\sqrt{2}\pi^{N/2} N^N} \sum_{r=0}^{R_0-1} z^{r/2} \int_{\R^{N-1}} \int_{u_\beta+\kappa\sqrt{z}}^\infty P_r(\bm{u_{[1]}},u_\alpha) du_\alpha d\bm{u_{[1]}}\\
+ e^{\frac{\pi^2N}{12z}-\frac{\kappa^2z}2} \sum_{r=0}^{S(R_0)-1} W_r(\kappa) z^{(r+1)/2} + O\Brb{e^{\frac{\pi^2N}{12\varepsilon}}\varepsilon^L}.
\end{multline*}
\end{prp}

We prove \Cref{prp:g1_asymp} in a series of lemmas, most of which are analogous to those found in \cite[Section 5]{BMRS24AP}; the corresponding proofs are thus omitted. Let $\mc T_{\bm\ell} = \mc T_{\alpha,\beta;N;c;\bm\ell}$ be the bijective image of $\mc S_{\alpha,\beta;N;c;\bm\ell}$ under the map $\bm m \mapsto \bm u$ given in \eqref{eq:uj_def}, and $\mc U_{\varepsilon,\lambda}$ the bijective image of $\mc N_{\varepsilon,\lambda}$ under the same map. Since $\bm m\mapsto \bm u$ takes $\R^N$ to $(-\log(2)/\sqrt{z}+\R\sqrt{z})^N$, we have $\mc T_{\bm\ell} \sbe (-\log(2)/\sqrt{z}+\R\sqrt{z})^N \sbe \C^N$. 

\begin{lem}[cf. {\cite[Lemma 5.2]{BMRS24AP}}]\label{lem:out_of_range}
If $y\ll \varepsilon^{1+\lambda+\delta}$ for some $\delta>0$, and $Q(\bm u)$ is a polynomial in $\bm u$, then as $z\to 0$ we have
\[
\sum_{\bm u \in \mc T_{\bm\ell}\bs\mc U_{\varepsilon,\lambda}} \vb{Q(\bm u) e^{-\bm u^T\bm u}} = O\rb{\varepsilon^L}, \quad \text{ and } \quad \int_{(-\log(2)/\sqrt{z}+\R\sqrt{z})^N\bs\mc U_{\varepsilon,\lambda}} \vb{Q(\bm u) e^{-\bm u^T\bm u}} d\bm u = O\rb{\varepsilon^L}.
\]
\end{lem}
\begin{rmk}
The only important ingredient in \Cref{lem:out_of_range} is that the summand (resp. integrand) $Q(\bm u) e^{-\bm u^T\bm u}$ has exponential decay outside $\mc U_{\varepsilon,\lambda}$, and it makes essentially no difference whether we sum or integrate. In fact, with the same proof as in \cite{BMRS24AP}, we may sum over some of the variables and integrate over the other variables, and obtain the same bound.
\end{rmk}

\begin{lem}[{\cite[Lemma 5.3]{BMRS24AP}}]\label{lem:exponential_bound}
If $y\ll \varepsilon^{1+\lambda+\delta}$ for some $\delta>0$, and $\bm u \in \mc U_{\varepsilon,\lambda}$, then as $z\to 0$ we have
\[
\Bvb{e^{\frac{\pi^2N}{12z}-\bm u^T\bm u}} \le e^{\frac{\pi^2N}{12\varepsilon}}.
\]
\end{lem}

\begin{lem}[{\cite[Lemma 5.4]{BMRS24AP}}]\label{lem:g1_sum_asymp}
If $y\ll \varepsilon^{1+\lambda+\delta}$ for some $\delta>0$, then for $R\in\N$ we have, as $z\to 0$,
\[
g_{\bm\ell}^{[1]}(z) = \rb{\frac z\pi}^{N/2} \frac{1}{\sqrt{2}} \sum_{r=0}^{R-1} \sum_{\bm u \in \mc T_{\bm\ell} \cap \mc U_{\varepsilon,\lambda}} e^{\frac{\pi^2N}{12z}} P_r(\bm u) z^{r/2} + O\Brb{e^{\frac{\pi^2N}{12\varepsilon}}\varepsilon^{N\rb{\lambda+1/2}+3R\rb{\lambda+2/3}}}.
\]
\end{lem}

The next lemma estimates the inner sum over $\bm u$ in \Cref{lem:g1_sum_asymp}.

\begin{lem}\label{lem:P_asymp}
Suppose $y\ll \varepsilon^{1+\lambda+\delta}$ for some $\delta>0$. For $L\in\N$, define $R_1 = R_1(r,L) := \big\lceil\frac{L/2-(\lambda+1/2)(3r+N+1)}{2\lambda+2}\big\rceil$. Then as $z\to 0$ we have 
\begin{multline*}
\sum_{\bm u \in \mc T_{\bm\ell} \cap \mc U_{\varepsilon,\lambda}} e^{\frac{\pi^2N}{12z}} P_r(\bm u) = \frac{1}{z^{N/2} N^N} \int_{\R^{N-1}} \int_{u_\beta+\kappa\sqrt{z}}^\infty e^{\frac{\pi^2N}{12z}} P_r(\bm{u_{[1]}},u_\alpha) du_\alpha d\bm{u_{[1]}}\\
+ e^{\frac{\pi^2N}{12z}-\frac{\kappa^2z}2} \sum_{j=0}^{3r+2R_1-1} V_{r,j}(\kappa) z^{(j+1-N)/2} + O\Brb{e^{\frac{\pi^2N}{12\varepsilon}} \varepsilon^{(L+1-N)/2}}.
\end{multline*}
for some polynomials $V_{r,j}$ of degree $\le j$.
\end{lem}
\begin{proof}
Swapping $\alpha$ and $\beta$ when necessary, we may assume $c\ge 0$. For the first step, we consider the sum over $u_\alpha$. We decompose
\begin{equation}\label{eq:P_sum_decomp}
\sum_{\bm u \in \mc T_{\bm\ell} \cap \mc U_{\varepsilon,\lambda}} e^{\frac{\pi^2N}{12z}} P_r(\bm u) = \sum_{\substack{\bm{u_{[1]}}\in\C^{N-1} \\ \exists u_\alpha \in\C, (\bm{u_{[1]}},u_\alpha)\in \mc T_{\bm\ell}\cap \mc U_{\varepsilon,\lambda}}} \sum_{\substack{u_\alpha\in \mf u(\bm{u_{[1]}}) \\ (\bm{u_{[1]}},u_\alpha)\in\mc U_{\varepsilon,\lambda}}} e^{\frac{\pi^2N}{12z}} P_r(\bm{u_{[1]}},u_\alpha),
\end{equation}
where $\mf u(\bm{u_{[1]}}) := \cbm{u_\beta+t\sqrt{z}}{t \in \kappa + N\N_0}$. If $\bm{u_{[1]}}\in\C^{N-1}$ is such that there exists $u_\alpha\in\C$ with $(\bm{u_{[1]}},u_\alpha)\in\mc T_{\bm\ell} \cap \mc U_{\varepsilon,\lambda}$, then we have $\mf u(\bm{u_{[1]}}) = \cbm{u_\alpha\in\C}{(\bm{u_{[1]}},u_\alpha)\in\mc T_{\bm\ell}}$, so the decomposition \eqref{eq:P_sum_decomp} makes sense. By \Cref{lem:out_of_range}, we may extend the inner sum in \eqref{eq:P_sum_decomp} by removing the condition $(\bm{u_{[1]}},u_\alpha) \in\mc U_{\varepsilon,\lambda}$. This introduces an error term of size $O(e^{\pi^2N/12\varepsilon}\varepsilon^L)$ for all $L\in\N$, which is negligible for our purpose. Applying \Cref{prp:EMF} with $a = u_\beta+\kappa\sqrt{z}$, $z\mapsto N\sqrt{z}$ gives
\begin{multline}\label{eq:P_sum_EMF_crude}
\sum_{u_\alpha \in \mf u(\bm{u_{[1]}})} e^{\frac{\pi^2N}{12z}} P_r(\bm{u_{[1]}},u_\alpha) = \frac{1}{N\sqrt{z}} \int_{u_\beta+\kappa\sqrt{z}}^{u_\beta+\infty\sqrt{z}} e^{\frac{\pi^2N}{12z}} P_r(\bm{u_{[1]}},u_\alpha) du_\alpha\\ + \frac{1}2 e^{\frac{\pi^2N}{12z}} P_r\rb{\bm{u_{[1]}}, u_\beta+\kappa\sqrt{z}} -\sum_{j=1}^R \frac{B_{2j} N^{2j-1} z^{j-1/2}}{(2j)!} e^{\frac{\pi^2N}{12z}} P_r^{(2j-1)}\rb{\bm{u_{[1]}},u_\beta+\kappa\sqrt{z}}\\
+ O(1)z^R \int_{u_\beta+\kappa\sqrt{z}}^{u_\beta+\infty\sqrt{z}} \Bvb{e^{\frac{\pi^2N}{12z}} P_r^{(2R+1)}(\bm{u_{[1]}},u_\alpha)} du_\alpha,
\end{multline}
where the derivative is taken with respect to $u_\alpha$.

First we estimate the error term in \eqref{eq:P_sum_EMF_crude}. By a direct calculation, we see that for $k\in\N$ we have
\begin{equation}\label{eq:C_derivative_def}
P_r^{(k)}(\bm u) = C_{r,k}(\bm u) e^{-\bm u^T\bm u},
\end{equation}
for some polynomial $C_{r,k}(\bm u)$ of degree at most $3r+k$. By \Cref{lem:out_of_range}, we may restrict the second integral in \eqref{eq:P_sum_EMF_crude} to $\mc U_{\varepsilon,\lambda}$, introducing an error of size $O(e^{\pi^2N/12\varepsilon}\varepsilon^L)$ for all $L\in\N$. Now we consider the part of the integral over $\mc U_{\varepsilon,\lambda}$, which has measure $\ll \varepsilon^{\lambda+1/2}$. For $\bm u \in \mc U_{\varepsilon,\lambda}$, we use \Cref{lem:exponential_bound} and that $\vb{\bm u}\ll \varepsilon^{\lambda+1/2}$ to conclude that
\begin{equation*}
e^{\frac{\pi^2N}{12z}} P_r^{(k)}(\bm u) \ll e^{\frac{\pi^2N}{12\varepsilon}} \varepsilon^{\rb{\lambda+1/2}(3r+k)}.
\end{equation*}
Hence we have, after restricting the integral to over $\mc U_{\varepsilon,\lambda}$, the bound
\[
\int_{u_\beta+\kappa\sqrt{z}}^{u_\beta+\infty\sqrt{z}} \Bvb{e^{\frac{\pi^2N}{12z}} P_r^{(2R+1)}(\bm{u_{[1]}},u_\alpha)} du_\alpha \ll e^{\frac{\pi^2N}{12\varepsilon}} \varepsilon^{\rb{\lambda+1/2}(3r+2R+2)},
\]
and thus we have as $z\to 0$
\begin{multline}\label{eq:P_sum_EMF_refined}
\sum_{u_\alpha \in \mf u(\bm{u_{[1]}})} e^{\frac{\pi^2N}{12z}} P_r(\bm{u_{[1]}},u_\alpha) = \frac{1}{N\sqrt{z}} \int_{u_\beta+\kappa\sqrt{z}}^{u_\beta+\infty\sqrt{z}} e^{\frac{\pi^2N}{12z}} P_r(\bm{u_{[1]}},u_\alpha) du_\alpha + \frac 12 e^{\frac{\pi^2N}{12z}} P_r\rb{\bm{u_{[1]}}, u_\beta+\kappa\sqrt{z}}\\
-\sum_{j=1}^R \frac{B_{2j} N^{2j-1} z^{j-1/2}}{(2j)!} e^{\frac{\pi^2N}{12z}} P_r^{(2j-1)}\rb{\bm{u_{[1]}},u_\beta+\kappa\sqrt{z}} + O\Brb{e^{\frac{\pi^2N}{12\varepsilon}} \varepsilon^{\rb{\lambda+1/2}(3r+2R+2)+R}}.
\end{multline}

Next we consider the sum over the other variables. Let $1\le \gamma \le N$, $\gamma \ne \alpha$, and consider the summation over $u_\gamma$. To emphasise the $\gamma$-th and the $\alpha$-th entries of $\bm u$, we write $\bm{u_{[2]}}$ to denote the remaining $N-2$ components of $\bm u$, and write $\bm u = (\bm{u_{[2]}}, u_\gamma, u_\alpha)$. Again, this does not say that $u_\gamma, u_\alpha$ are the final components of $\bm u$. We also write $\bm{u_{[1]}} = (\bm{u_{[2]}}, u_\gamma)$. Splitting off the sum over $u_\gamma$ and $u_\alpha$, we write
\[
\sum_{\bm u \in \mc T_{\bm\ell} \cap \mc U_{\varepsilon,\lambda}} e^{\frac{\pi^2N}{12z}} P_r(\bm u) = \sum_{\substack{\bm{u_{[2]}} \in \C^{N-2} \\ \exists u_\gamma, u_\alpha\in\C, \; (\bm{u_{[2]}},u_\gamma,u_\alpha) \in \mc T_{\bm\ell} \cap \mc U_{\varepsilon,\lambda}}} \sum_{\substack{u_\gamma \in \mf u_\gamma, \; u_\alpha \in \mf u(\bm{u_{[2]}}, u_\gamma) \\ (\bm{u_{[2]}}, u_\gamma, u_\alpha) \in \mc U_{\varepsilon,\lambda}}} e^{\frac{\pi^2N}{12z}} P_r(\bm{u_{[2]}},u_\gamma,u_\alpha),
\]
where $\mf u_\gamma := \cbm{-\log(2)/\sqrt{z} + t\sqrt{z}}{t\in \ell_\gamma + N\N_0}$. Analogous to the summation over $u_\alpha$, we may use \Cref{lem:out_of_range} to extend the inner sum by removing the condition $(\bm{u_{[2]}},u_\gamma,u_\alpha) \in \mc U_{\varepsilon,\lambda}$, introducing an error of size $O(e^{\pi^2N/12\varepsilon} \varepsilon^L)$ for all $L\in\N$, which is negligible. So we may consider instead the sum
\[
\sum_{\substack{u_\gamma \in \mf u_\gamma \\ u_\alpha \in \mf u_\alpha (\bm{u_{[2]}}, u_\gamma)}} e^{\frac{\pi^2N}{12z}}  P_r(\bm{u_{[2]}}, u_\gamma, u_\alpha) = \sum_{u_\gamma \in \mf u_\gamma} F_r(\bm{u_{[2]}}, u_\gamma), \quad F_r (\bm{u_{[1]}}) := \sum_{u_\alpha\in\mf u(\bm{u_{[1]}})} e^{\frac{\pi^2N}{12z}} P_r(\bm{u_{[1]}},u_\alpha).
\]

Again, we apply \Cref{prp:EMF}, and write
\begin{multline}\label{eq:P_second_sum_EMF}
\sum_{u_\gamma\in\mf u_\gamma} F_r(\bm{u_{[2]}}, u_\gamma) = \frac{1}{N\sqrt{z}} \int_{-\log(2)/\sqrt{z}+\ell_\gamma\sqrt{z}}^{-\log(2)/\sqrt{z}+\infty\sqrt{z}} F_r(\bm{u_{[2]}}, u_\gamma) du_\gamma\\
+ \frac 12 F_r\rb{\bm{u_{[2]}},-\tfrac{\log(2)}{\sqrt{z}}+\ell_\gamma\sqrt{z}} - \sum_{j=1}^R \frac{B_{2j} N^{2j-1} z^{j-1/2}}{(2j)!} F_r^{(2j-1)} \rb{\bm{u_{[2]}},-\tfrac{\log(2)}{\sqrt{z}}+\ell_\gamma\sqrt{z}}\\
+ O(1) z^R  \int_{-\log(2)/\sqrt{z}+\ell_\gamma\sqrt{z}}^{-\log(2)/\sqrt{z}+\infty\sqrt{z}} \vb{F_r^{(2R+1)}(\bm{u_{[2]}}, u_\gamma)} du_\gamma,
\end{multline}
where the derivative is taken with respect to $u_\gamma$. Thanks to the exponential decay of $e^{-\bm u^T\bm u}$ in the cone $y\ll \varepsilon^{1+\lambda+\delta}$, we have, as $z\to 0$ for all $j\in\N_0$ and $L\in\N$,
\begin{equation}\label{eq:P_second_sum_small}
\hspace{-0.1cm} F_r^{(j)}\rb{\bm{u_{[2]}},-\tfrac{\log(2)}{\sqrt{z}}+\ell_\gamma\sqrt{z}} = O\Big(e^{\frac{\pi^2N}{12\varepsilon}} \varepsilon^L\Big),\  \int_{-\log(2)/\sqrt{z}-\infty\sqrt{z}}^{-\log(2)/\sqrt{z}-\ell_\gamma\sqrt{z}} F_r(\bm{u_{[2]}}, u_\gamma)du_\gamma = O\Big(e^{\frac{\pi^2N}{12\varepsilon}} \varepsilon^L\Big).
\end{equation}
Therefore, the lower boundary of the first integral in \eqref{eq:P_second_sum_EMF} can be extended to $-\log(2)/\sqrt{z}-\infty\sqrt{z}$, and all the terms except the main term can be ignored, introducing an error term of size $O(e^{\pi^2N/12\varepsilon} \varepsilon^L)$ for all $L\in\N$. Meanwhile, the error term in \eqref{eq:P_second_sum_EMF} has size
\begin{equation}\label{eq:P_second_sum_error_bound}
z^R \int_{-\log(2)/\sqrt{z}+\ell_\gamma\sqrt{z}}^{-\log(2)/\sqrt{z}+\infty\sqrt{z}} \vb{F_r^{(2R+1)}(\bm{u_{[2]}}, u_\gamma)} du_\gamma = O\Brb{e^{\frac{\pi^2N}{12\varepsilon}} \varepsilon^{\rb{\lambda+1/2}(3r+2R+2)+R+\lambda}}.
\end{equation}
By taking $R$ sufficiently large, it follows from \eqref{eq:P_second_sum_EMF}, \eqref{eq:P_second_sum_small}, and \eqref{eq:P_second_sum_error_bound} that
\[
\sum_{u_\gamma\in\mf u_\gamma} F_r(\bm{u_{[2]}}, u_\gamma) = \frac{1}{N\sqrt{z}} \int_{-\log(2)/\sqrt{z}+\R\sqrt{z}} F_r(\bm{u_{[2]}}, u_\gamma)du_\gamma + O\Brb{e^{\frac{\pi^2N}{12\varepsilon}} \varepsilon^L}
\]
for all $L\in\N$. Using the same argument, we sum over the other coordinates, and obtain that
\[
\sum_{\bm u \in \mc T_{\bm\ell} \cap \mc U_{\varepsilon,\lambda}} e^{\frac{\pi^2N}{12z}} P_r(\bm u) = \frac{z^{(1-N)/2}}{N^{N-1}} \int_{\rb{-\log(2)/\sqrt{z}+\R\sqrt{z}}^{N-1}} F_r(\bm{u_{[1]}}) d\bm{u_{[1]}} + O\Brb{e^{\frac{\pi^2N}{12\varepsilon}} \varepsilon^L}
\]
for all $L\in\N$. Applying a mixed version of \Cref{lem:out_of_range} (see the remark below \Cref{lem:out_of_range}), we may restrict the integral above to the set (see \eqref{eq:N_def} for the definition of $\bm\mu$)
\[
\mc U_{\bm{[1]}} := \cbm{\bm u_{[1]} \in \rb{-\tfrac{\log(2)}{\sqrt{z}}+\R\sqrt{z}}^{N-1}}{\vb{\bm{\mu_{[1]}}}\le \varepsilon^\lambda},
\]
again with an error of size $O(e^{\pi^2N/12\varepsilon} \varepsilon^L)$ for all $L\in\N$, which is negligible. So we have
\[
\sum_{\bm u \in \mc T_{\bm\ell} \cap \mc U_{\varepsilon,\lambda}} e^{\frac{\pi^2N}{12z}}  P_r(\bm u) = \frac{z^{(1-N)/2}}{N^{N-1}} \int_{\mc U_{\bm{[1]}}} F_r(\bm{u_{[1]}}) d\bm{u_{[1]}} + O\Brb{e^{\frac{\pi^2N}{12\varepsilon}}\varepsilon^L}
\]
for all $L\in\N$. Applying the asymptotic formula \eqref{eq:P_sum_EMF_refined} gives
\begin{align}\nonumber
\sum_{\bm u \in \mc T_{\bm\ell} \cap \mc U_{\varepsilon,\lambda}} e^{\frac{\pi^2N}{12z}} P_r(\bm u) &= \frac{z^{-N/2}}{N^N} \int_{\mc U_{\bm{[1]}}} \int_{u_\beta+\kappa\sqrt{z}}^{u_\beta+\infty\sqrt{z}} e^{\frac{\pi^2N}{12z}} P_r(\bm{u_{[1]}},u_\alpha) du_\alpha d\bm{u_{[1]}}\\
&\hspace{0.5cm}+ \frac{z^{(1-N)/2}}{2N^{N-1}} \int_{\mc U_{\bm{[1]}}} e^{\frac{\pi^2N}{12z}} P_r(\bm{u_{[1]}},u_\beta+\kappa\sqrt{z}) d\bm{u_{[1]}}\nonumber\\
&\hspace{0.5cm}-\sum_{j=1}^R \frac{B_{2j} N^{2j-N} z^{j-N/2}}{(2j)!} \int_{\mc U_{\bm{[1]}}} e^{\frac{\pi^2N}{12z}} P_r^{(2j-1)}(\bm{u_{[1]}},u_\beta+\kappa\sqrt{z}) d\bm{u_{[1]}}\nonumber\\
&\hspace{0.5cm}+ O\Brb{e^{\frac{\pi^2N}{12\varepsilon}} \varepsilon^{\rb{\lambda+1/2}(3r+2R+2)+R+(1-N)/2}} \int_{\mc U_{\bm{[1]}}} d\bm{u_{[1]}}. \label{eq:P_all_sum_asymp1}
\end{align}
Since $\mc U_{\bm{[1]}}$ has measure $\ll\varepsilon^{(N-1)\rb{\lambda+1/2}}$, it follows that the error term in \eqref{eq:P_all_sum_asymp1} has size
\[
O\Brb{e^{\frac{\pi^2N}{12\varepsilon}} \varepsilon^{\rb{\lambda+1/2}(3r+2R+N+1)+R+(1-N)/2}}.
\]
For other terms in \eqref{eq:P_all_sum_asymp1}, we use \Cref{lem:out_of_range} to extend the integrals to $-\log(2)/\sqrt{z} + \R\sqrt{z}$, again introducing a negligible error. Moreover, since $P_r^{(j)}(\bm{u_{[1]}},u_\alpha)$ and all of its derivatives (with respect to $u_\gamma$) are holomorphic and have rapid decay, we may shift the path of integration and write
\begin{align}\nonumber
\sum_{\bm u \in \mc T_{\bm\ell} \cap \mc U_{\varepsilon,\lambda}} e^{\frac{\pi^2N}{12z}} P_r(\bm u) &= \frac{z^{-N/2}}{N^N} \int_{\R^{N-1}} \int_{u_\beta+\kappa\sqrt{z}}^{\infty} e^{\frac{\pi^2N}{12z}} P_r(\bm{u_{[1]}},u_\alpha) du_\alpha d\bm{u_{[1]}}\\
&\hspace{0.5cm}+ \frac{z^{(1-N)2}}{2N^{N-1}} \int_{\R^{N-1}} e^{\frac{\pi^2N}{12z}} P_r(\bm{u_{[1]}}, u_\beta+\kappa\sqrt{z}) d\bm{u_{[1]}}\nonumber\\
&\hspace{0.5cm}-\sum_{j=1}^R \frac{B_{2j} N^{2j-N} z^{j-N/2}}{(2j)!} \int_{\R^{N-1}} e^{\frac{\pi^2N}{12z}} P_r^{(2j-1)}(\bm{u_{[1]}},u_\beta+\kappa\sqrt{z}) d\bm{u_{[1]}}\nonumber\\
&\hspace{0.5cm}+ O\Brb{e^{\frac{\pi^2N}{12\varepsilon}} \varepsilon^{\rb{\lambda+1/2}(3r+2R+N+1)+R+(1-N)/2}}.\label{eq:P_all_sum_asymp2}
\end{align}
To evaluate \eqref{eq:P_all_sum_asymp2}, we have to treat the integrals of the form 
\[
\int_{\R^{N-1}} e^{\frac{\pi^2N}{12z}} P_r^{(j)} (\bm{u_{[1]}},u_\beta+\kappa\sqrt{z}) d\bm{u_{[1]}}.
\]
Using \eqref{eq:C_derivative_def}, it follows that
\[
\int_{\R^{N-1}} e^{\frac{\pi^2N}{12z}} P_r^{(j)} \rb{\bm{u_{[1]}},u_\beta+\kappa\sqrt{z}} d\bm{u_{[1]}} = \int_{\R^{N-1}} e^{\frac{\pi^2N}{12z}} C_{r,j}\rb{\bm{u_{[1]}},u_\beta+\kappa\sqrt{z}} e^{-\bm{u_{[1]}}^T\bm{u_{[1]}} - (u_\beta+\kappa\sqrt{z})^2} d\bm{u_{[1]}}.
\]
Define $C^*_{r,j}(\bm{u_{[2]}},u_\beta,u_\alpha) := C_{r,j}(\bm{u_{[2]}}, u_\beta-u_\alpha/2, u_\beta+u_\alpha/2)$, where $\bm{u_{[2]}}$ denotes the remaining $N-2$ entries of $\bm u$. Then we have 
\[
C^*_{r,j}\rb{\bm{u_{[2]}}, u_\beta+\tfrac{\kappa\sqrt{z}}2, \kappa\sqrt{z}} = C_{r,j}(\bm{u_{[2]}},u_\beta,u_\beta+\kappa\sqrt{z}).
\]
We may then write (after a change of variables $u_\beta+\kappa\sqrt{z}/2 \mapsto u_\beta$)
\[
\int_{\R^{N-1}} e^{\frac{\pi^2N}{12z}} P_r^{(j)} (\bm{u_{[1]}},u_\beta+\kappa\sqrt{z}) d\bm{u_{[1]}} = e^{\frac{\pi^2N}{12z}-\frac{\kappa^2z}2} \int_{\R^{N-1}} C^*_{r,j} \rb{\bm{u_{[1]}},\kappa\sqrt{z}} e^{-\bm{u_{[1]}}^T\bm{u_{[1]}} - u_\beta^2} d\bm{u_{[1]}},
\]
with the inner integral being a linear combination of (multivariate) Gaussian integrals, with some power of $\kappa\sqrt{z}$. Noting that $C^*_{r,j}(\bm u)$ also has degree at most $3r+j$, we may write
\begin{multline*}
\frac {z^{(1-N)/2}}{2N^{N-1}} \int_{\R^{N-1}} e^{\frac{\pi^2N}{12z}} P_r(\bm{u_{[1]}},u_\beta+\kappa\sqrt{z}) d\bm{u_{[1]}}\\
-\sum_{j=1}^R \frac{B_{2j} N^{2j-N} z^{j-N/2}}{(2j)!} \int_{\R^{N-1}} \hspace{-0.9em} e^{\frac{\pi^2N}{12z}} P_r^{(2j-1)}(\bm{u_{[1]}},u_\beta+\kappa\sqrt{z}) d\bm{u_{[1]}} = e^{\frac{\pi^2N}{12z}-\frac{\kappa^2z}2} \hspace{-0.6em} \sum_{j=0}^{3r+2R-1} \hspace{-0.6em} V_{r,j}(\kappa) z^{(j+1-N)/2},
\end{multline*}
where $V_{r,j}$ are polynomials of degree $\le j$. Hence we may write
\begin{multline}\label{eq:P_pre_asymp}
\sum_{\bm u \in \mc T_{\bm\ell} \cap \mc U_{\varepsilon,\lambda}} e^{\frac{\pi^2N}{12z}} P_r(\bm u) = \frac{z^{-N/2}}{N^N} \int_{\R^{N-1}} \int_{u_\beta+\kappa\sqrt{z}}^\infty e^{\frac{\pi^2N}{12z}} P_r(\bm{u_{[1]}},u_\alpha) du_\alpha d\bm{u_{[1]}}\\
+ e^{\frac{\pi^2N}{12z}-\frac{\kappa^2z}2} \sum_{j=0}^{3r+2R-1} V_{r,j}(\kappa) z^{(j+1-N)/2} + O\Brb{e^{\frac{\pi^2N}{12\varepsilon}} \varepsilon^{\rb{\lambda+1/2}(3r+2R+N+1)+R+(1-N)/2}}.
\end{multline}
For $L\in\N$, we take $R = R_1(r,L) = \big\lceil\frac{L/2-(\lambda+1/2)(3r+N+1)}{2\lambda+2}\big\rceil$, then the error term in \eqref{eq:P_pre_asymp} has size $O(e^{\pi^2N/12\varepsilon} \varepsilon^{(L+1-N)/2})$, proving the lemma.
\end{proof}

\begin{proof}[Proof of \Cref{prp:g1_asymp}]
We estimate the following expression, occurring in \Cref{lem:g1_sum_asymp}:
\[
\rb{\frac z\pi}^{N/2} \frac{1}{\sqrt{2}} \sum_{r=0}^{R-1} \sum_{\bm u \in \mc T_{\bm\ell} \cap \mc U_{\varepsilon,\lambda}} e^{\frac{\pi^2N}{12z}} P_r(\bm u) z^{r/2} = \frac{1}{\sqrt{2}\pi^{N/2}} \sum_{r=0}^{R-1} z^{(N+r)/2} \sum_{\bm u \in \mc T_{\bm\ell} \cap \mc U_{\varepsilon,\lambda}} e^{\frac{\pi^2N}{12z}} P_r(\bm u).
\]
By \Cref{lem:P_asymp}, we have
\begin{multline*}
z^{(N+r)/2} \sum_{\bm u \in \mc T_{\bm\ell}\cap \mc U_{\varepsilon,\lambda}} e^{\frac{\pi^2N}{12z}} P_r(\bm u) = \frac{z^{r/2}}{N^N} \int_{\R^{N-1}} \int_{u_\beta+\kappa\sqrt{z}}^\infty e^{\frac{\pi^2N}{12z}} P_r(\bm{u_{[1]}},u_\alpha)du_\alpha d\bm{u_{[1]}}\\
+ e^{\frac{\pi^2N}{12z}-\frac{\kappa^2z}2} \sum_{j=0}^{3r+2R_1(r,R-r)-1} V_{r,j}(\kappa) z^{(r+j+1)/2} + O\Brb{e^{\frac{\pi^2N}{12\varepsilon}}\varepsilon^{(R+1)/2}}.
\end{multline*}
Summing over $r$ gives
\begin{multline*}
\sum_{r=0}^{R-1} z^{(N+r)/2} \sum_{\bm u \in \mc T_{\bm\ell}\cap \mc U_{\varepsilon,\lambda}} e^{\frac{\pi^2N}{12z}} P_r(\bm u) = \sum_{r=0}^{R-1} \frac{z^{r/2}}{N^N} \int_{\R^{N-1}} \int_{u_\beta+\kappa\sqrt{z}}^\infty e^{\frac{\pi^2N}{12z}} P_r(\bm{u_{[1]}},u_\alpha)du_\alpha d\bm{u_{[1]}}\\
+ e^{\frac{\pi^2N}{12z}-\frac{\kappa^2z}2} \sum_{r=0}^{R-1} \sum_{j=0}^{3r+2R_1(r,R-r)-1} V_{r,j}(\kappa) z^{(r+j+1)/2} + O\Brb{e^{\frac{\pi^2N}{12\varepsilon}}\varepsilon^{(R+1)/2}}.
\end{multline*}
Combining the terms with same power in $z$, we get 
\begin{align*}
&\frac{1}{\sqrt{2}\pi^{N/2}} \sum_{r=0}^{R-1} z^{(N+r)/2} \sum_{\bm u \in \mc T_{\alpha,\beta;N;c;\bm\ell}\cap \mc U_{\varepsilon,\lambda}} e^{\frac{\pi^2N}{12z}} P_r(\bm u)\\
&= \frac{1}{\sqrt{2}\pi^{N/2} N^N} \sum_{r=0}^{R-1} z^{r/2} \int_{\R^{N-1}} \int_{u_\beta+\kappa\sqrt{z}}^\infty e^{\frac{\pi^2N}{12z}} P_r(\bm{u_{[1]}},u_\alpha)du_\alpha d\bm{u_{[1]}}\\
&\hspace{1cm}+ e^{\frac{\pi^2N}{12z}-\frac{\kappa^2z}2} \sum_{r=0}^{S(R)-1} W_r(\kappa) z^{(r+1)/2} + O\Brb{e^{\frac{\pi^2N}{12\varepsilon}}\varepsilon^{(R+1)/2}},
\end{align*}
where $S(R) := 4(R-1)+2R_1(R-1,1) = 4(R-1) + 2\big\lceil\frac{1/2 - (\lambda+1/2)(3R+N-2)}{2\lambda+2}\big\rceil$, and $W_r$ are polynomials of degree $\le r$. Since $-\frac 23<\lambda < -\frac 12$, we see that the error term $O(e^{\pi^2N/12\varepsilon} \varepsilon^{N\rb{\lambda+1/2} + 3R\rb{\lambda+2/3}})$ from \Cref{lem:g1_sum_asymp} dominates, and the proposition is established.
\end{proof}

Now we show that $g_{\bm\ell}^{[1]}(z)$ is actually small for larger values of $\vb{y}$. We require the following technical lemma about $\Lambda(y)$ (see \eqref{eq:Lambda_def}), from \cite{BMRS24AP}.

\begin{lem}[{\cite[Lemma 5.6]{BMRS24AP}}]\label{lem:sy}
Let $s(y) := \re\big(\frac{\Lambda(y)}{1+iy} - \frac{\pi^2N}{12}\big)$. Then we have the following:
\begin{enumerate}
\item We have $s(y)\le 0$ for all $y\in\R$, and the equality holds if and only if $y=0$.
\item For any $y_0>0$, there exists $d>0$ such that $s(y)<-d$ for $\vb{y}\ge y_0$.
\item We have, as $y\to 0$,
\[
s(y) = N\rb{\log(2)^2-\frac{\pi^2}{12}}y^2 + O\rb{y^4}.
\]
\end{enumerate}
\end{lem}

\begin{prp}\label{prp:g1_wide_range}
Suppose $\varepsilon^{1/2-\delta} \ll y \ll \varepsilon^{-1-3\lambda/2+\delta}$ for some $\delta>0$. Then as $z\to 0$ we have
\[
g_{\bm\ell}^{[1]}(z) \ll e^{\frac{\pi^2N}{12\varepsilon}} \varepsilon^L \quad \fa L\in\N.
\]
\end{prp}
\begin{proof}
It suffices to show that 
\begin{equation}\label{eq:g1_exponent_small}
e^{-\frac{\pi^2N}{12\varepsilon}} g_{\bm\ell}^{[1]}(z) \ll \varepsilon^L \quad \fa L\in\N.
\end{equation}
From \Cref{prp:wide_asymp}, if $y\ll \varepsilon^{-1-3\lambda/2+\delta}$ we have for $R\in\N$ as $z\to 0$ an expansion
\begin{multline}\label{eq:g1_wide_range_asymp}
e^{-\frac{\pi^2N}{12\varepsilon}} g_{\bm\ell}^{[1]}(z) = 2^{-(1+iy)/2} \rb{1-2^{-1(1+iy)}}^{-N/2} \rb{\frac{z}{2\pi}}^{N/2} e^{\frac 1\varepsilon (\frac{\Lambda(y)}{1+iy} - \frac{\pi^2N}{12})}\\
\hspace{-0.1cm}\times\hspace{-0.5em} \sum_{\bm m\in\mc S_{\bm\ell} \cap \mc N_{\varepsilon,\lambda}} \hspace{-0.5em} e^{-\frac12\bm v^T\mc A\bm v+\frac{1}{\sqrt z}\rb{-\log(2)(1+iy)-\log(1-2^{-(1+iy)})}\sum_{j=1}^N v_j} \rb{\sum_{r=0}^{R-1} D_r(\bm v,y)z^{r/2} + O\rb{\varepsilon^{2R\delta_2}}}. \hspace{-0.1cm}
\end{multline}
To prove \eqref{eq:g1_exponent_small}, it suffices to show that the exponent in \eqref{eq:g1_wide_range_asymp}
\[
\frac 1\varepsilon \rb{\frac{\Lambda(y)}{1+iy} - \frac{\pi^2N}{12}} - \frac 12 \bm v^T\mc A\bm v + \frac{1}{\sqrt{z}} \rb{-\log(2)(1+iy) - \log\rb{1-2^{-(1+iy)}}} \sum_{j=1}^N v_j
\]
has negative real part of size $\gg \varepsilon^{-\delta_0}$ for some $\delta_0>0$; we note that the other factors in \eqref{eq:g1_wide_range_asymp} are bounded as $z\to 0$, and the sum over $\bm m$ contains $\ll \varepsilon^{N\lambda}$ terms. By \Cref{lem:sy}(1) the real part of $\frac 1\varepsilon \re\big(\frac{\Lambda(y)}{1+iy}-\frac{\pi^2N}{12}\big) = \frac{s(y)}\varepsilon$ is negative. So it suffices to show that this exponent has sufficiently large size and dominates other exponents with positive real parts.

First suppose $\varepsilon^{1/2-\delta}\ll y \ll 1$. By \Cref{lem:sy}(3), we find that
\[
\frac{s(y)}{\varepsilon} \gg \frac{y^2}{\varepsilon} \gg \varepsilon^{-2\delta}
\]
as $\varepsilon \to 0$. Meanwhile, by computing the Taylor series expansion we have, for $y\ll 1$,
\[
\re\rb{-\log(2)(1+iy) - \log\rb{1-2^{-(1+iy)}}} \ll y^2.
\]
Since $\bm m \in\mc S_{\bm\ell} \cap \mc N_{\varepsilon,\lambda}$, we have $v_j/\sqrt{z} \ll \varepsilon^\lambda$ (see the remark below \Cref{prp:wide_asymp}), we conclude that
\[
\re\rb{\frac{1}{\sqrt{z}} \rb{-\log(2)(1+iy) - \log\rb{1-2^{-(1+iy)}}}\sum_{j=1}^N v_j} \ll \varepsilon^\lambda y^2.
\]
Next we consider the exponent $-\frac 12 \bm v^T\mc A\bm v$. For this we split into two subcases.
\begin{enumerate}
\item Suppose $\varepsilon^{1/2-\delta} \ll y \ll \varepsilon^{1/4}$. Since $v_j$ is a real multiple of $\sqrt{z}$, the condition $y\ll \varepsilon^{1/4}$ implies that $-\frac 12\bm v^T\mc A\bm v$ has negative real part as $z\to 0$.
\item Suppose $\varepsilon^{1/4} \ll y \ll 1$. In this case we have $\vb{\bm v} \ll \varepsilon^{1/2+\lambda}$; this follows from remark below \Cref{prp:wide_asymp}, using that $z\asymp \varepsilon$ (which holds for $\varepsilon^{1/4} \ll y \ll 1$). It then follows that $-\frac 12\bm v^T\mc A \bm v\ll \varepsilon^{1+2\lambda}$. 
\end{enumerate}
In either subcase, the exponent $s(y)/\varepsilon$ has size $\gg \varepsilon^{-2\delta}$ and dominates other exponents appearing in \eqref{eq:g1_wide_range_asymp} that have positive real parts, proving \eqref{eq:g1_exponent_small} for the case $\varepsilon^{1/2-\delta}\ll y \ll 1$.

Now we consider the case $1\ll y \ll \varepsilon^{-1-3\lambda/2+\delta}$. By \Cref{lem:sy}(2), we have $s(y)/\varepsilon \gg \varepsilon^{-1}$. Meanwhile, using the trivial bound
\[
\re\rb{-\log(2)(1+iy) - \log\rb{1-2^{-(1+iy)}}} \ll 1,
\]
it follows from the bound $v_j/\sqrt{z} \le \varepsilon^\lambda$ that
\[
\re\rb{\frac{1}{\sqrt{z}} \rb{-\log(2)(1+iy) - \log\rb{1-2^{-(1+iy)}}} \sum_{j=1}^N v_j} \ll \varepsilon^\lambda.
\]
Finally, we use the bound $\vb{\bm v} \ll \varepsilon^{\lambda/4+\delta_2/2}$ (see the remark below \Cref{prp:wide_asymp}) and deduce that
\[
-\frac 12\bm v^T\mc A\bm v \ll \varepsilon^{\lambda/2+\delta_2}.
\]
So \eqref{eq:g1_exponent_small} also holds for the case $1\ll y \ll \varepsilon^{-1-3\lambda/2+\delta}$. This finishes the proof of the proposition.
\end{proof}

\section{Constructing the asymptotic expansion}\label{section:construction}

In this section, we establish the existence of an asymptotic expansion for $d_{\alpha,\beta;N;c_0n^{1/4}}(n)$ in an explicit manner, using the circle method. We often make use of the following lemma, which gives asymptotic expansions for integrals that arise when we apply the circle method.
\begin{lem}\label{lem:NR}{\cite[Lemma 3.7]{NR17IP}}
Let $A\ge 0$, $B>0$, and $R\in\N$. Then as $n\to\infty$ we have
\[
\frac{1}{2\pi i}\int_{\mc C_1\rb{\frac{B}{\sqrt{n}}}} z^Ae^{\frac{B^2}{z}+nz} dz = n^{(-2A-3)/4} e^{2B\sqrt{n}} \rb{\sum_{r=0}^{R-1} \frac{T_{A,B,r}}{n^{r/2}} + O\rb{n^{-R/2}}},
\]
where $\mc C_1(\eta) := \cbm{z = \eta(1+iy)}{\vb{y}\le \theta}$ for a fixed constant $\theta > 0$, and
\[
T_{A,B,r} := \frac{(-4B)^{-r} B^{A+1/2} \Gamma\rb{A+r+\frac 32}}{2\sqrt{\pi} r! \Gamma\rb{A-r+\frac 32}}.
\]
\end{lem}

Using Cauchy's theorem, we express $d_{\alpha,\beta;N;c_0n^{1/4}}(n)$ as an integral
\[
d_{\alpha,\beta;N;c_0n^{1/4}}(n) = \frac{1}{2\pi i} \int_{\mc C} \frac{D_{\alpha,\beta;N;c_0n^{1/4}}(q)}{q^{n+1}} dq,
\]
where $\mc C$ is a circle centred at the origin inside the unit circle surrounding zero exactly once counterclockwise. Using \eqref{eq:D_def} and the change of variables $q = e^{-z}$, we write, for any $\eta>0$
\begin{align*}
d_{\alpha,\beta;N;c_0n^{1/4}}(n) &= \frac 1N \sum_{k=0}^{N-1} \sum_{\bm\ell\in(\Z/N\Z)^N} \zeta_N^{(n-NH(\bm\ell))k} \frac{1}{2\pi i} \int_{\eta-\pi i}^{\eta+\pi i} g_{\alpha,\beta;N;c_0n^{1/4};\bm\ell}(z) e^{\frac{nz}N} dz\\
&= \sum_{\substack{\bm\ell\in(\Z/N\Z)^N\\ NH(\bm\ell) \equiv n \ppmod{N}}} \frac 1{2\pi i} \int_{\eta-\pi i}^{\eta+\pi i} g_{\alpha,\beta;N;c_0n^{1/4};\bm\ell}(z) e^{\frac{nz}N} dz.
\end{align*}
Throughout we shall pick $\eta = \eta(n) := \frac{\pi N}{2\sqrt{3n}}$. In other words, we take $\varepsilon = \eta$ for results from \Cref{section:asymp}. Using \eqref{eq:g_split}, we split
\[
d_{\alpha,\beta;N;c_0n^{1/4}}(n) = \sum_{\substack{\bm\ell\in(\Z/N\Z)^N\\ NH(\bm\ell) \equiv n \ppmod{N}}}\frac{1}{2\pi i} \int_{\eta-\pi i}^{\eta+\pi i} \rb{g_{\alpha,\beta;N;c_0n^{1/4};\bm\ell}^{[1]}(z)+g_{\alpha,\beta;N;c_0n^{1/4};\bm\ell}^{[2]}(z)} e^{\frac{nz}N} dz.
\]
By \Cref{prp:out_of_range}, as $n\to\infty$ we have
\[
\frac{1}{2\pi i} \int_{\eta-\pi i}^{\eta+\pi i} g_{\alpha,\beta;N;c_0n^{1/4};\bm\ell}^{[2]}(z) e^{\frac{nz}N} dz \ll \frac{1}{2\pi i} \int_{\eta-\pi i}^{\eta+\pi i} e^{\frac{\pi^2N}{12\eta}+\frac{n\eta}N} \eta^L dz \ll e^{\pi\sqrt{n/3}} n^{-L/2} \quad \fa L\in\N.
\]
Next we estimate the integral
\[
\frac 1{2\pi i} \int_{\eta-\pi i}^{\eta+\pi i} g_{\alpha,\beta;N;c_0n^{1/4};\bm\ell}^{[1]}(z) e^{\frac{nz}N} dz.
\]

Let $\theta>0$ be fixed. We split the integral into the \emph{major arc} $\mc C_1(\eta) := \cbm{z = \eta(1+iy)}{\vb{y}\le \theta}$ and the \emph{minor arc} $\mc C_2(\eta) := \cbm{z = \eta(1+iy)}{\theta < \vb{y} \le \pi/\eta}$. First we consider the minor arc integral
\[
\frac{1}{2\pi i} \int_{\mc C_2(\eta)} g_{\alpha,\beta;N;c_0n^{1/4};\bm\ell}^{[1]}(z) e^{\frac{nz}N} dz.
\]
By \Cref{prp:minor_arc} and the observation that $\vb{\mc N_{\varepsilon,\lambda}} \ll \varepsilon^{N\lambda}$, we have, as $\varepsilon = \re(z) \to 0$
\[
g_{\alpha,\beta;N;c_0n^{1/4};\bm\ell}^{[1]}(z) \ll e^{\frac{\pi^2N}{12\varepsilon}} \varepsilon^L \quad \fa L\in\N.
\]
So we have
\[
\frac{1}{2\pi i} \int_{\mc C_2(\eta)} g_{\alpha,\beta;N;c_0n^{1/4};\bm\ell}^{[1]}(z) e^{\frac{nz}N} dz \ll \frac{1}{2\pi i} \int_{\mc C_2(\eta)} e^{\frac{\pi^2N}{12\eta}+\frac{n\eta}N} \eta^L dz \ll e^{\pi\sqrt{n/3}} n^{-L/2} \quad \fa L\in\N.
\]

Now consider the major arc integral
\[
\frac{1}{2\pi i} \int_{\mc C_1(\eta)} g_{\alpha,\beta;N;c_0n^{1/4};\bm\ell}^{[1]}(z) e^{\frac{nz}N} dz.
\]
Let $Y = Y(\eta) := \eta^{\frac 12-\delta}$ for some fixed, sufficiently small $\delta > 0$. We split the integral
\begin{multline}\label{eq:major_arc_split}
\frac{1}{2\pi i} \int_{\mc C_1(\eta)} g_{\alpha,\beta;N;c_0n^{1/4};\bm\ell}^{[1]}(z) e^{\frac{nz}N} dz\\
= \frac{1}{2\pi i} \int_{\mc C_{1a}(\eta)} g_{\alpha,\beta;N;c_0n^{1/4};\bm\ell}^{[1]}(z) e^{\frac{nz}N} dz + \frac{1}{2\pi i} \int_{\mc C_{1b}(\eta)} g_{\alpha,\beta;N;c_0n^{1/4};\bm\ell}^{[1]}(z) e^{\frac{nz}N} dz,
\end{multline}
where $\mc C_{1a}(\eta) = \cbm{z = \eta(1+iy)}{\vb{y} \le Y}$, and $\mc C_{1b}(\eta) = \cbm{z = \eta(1+iy)}{Y<\vb{y}\le \theta}$. For the integral over $\mc C_{1b}(\eta)$, we use \Cref{prp:g1_wide_range} and deduce that
\[
\frac{1}{2\pi i} \int_{\mc C_{1b}(\eta)} g_{\alpha,\beta;N;c_0n^{1/4};\bm\ell}^{[1]}(z) e^{\frac{nz}N} dz \ll \int_{\mc C_{1b}(\eta)} e^{\frac{\pi^2N}{12\eta}+\frac{n\eta}N} \eta^L dz \ll e^{\pi\sqrt{n/3}} n^{-L/2} \quad \fa L\in\N.
\]

For the integral over $\mc C_{1a}(\eta)$ in \eqref{eq:major_arc_split}, we apply \Cref{prp:g1_asymp}, and write for $L\in\N$
\begin{multline}\label{eq:arc_1a_decomp}
\frac{1}{2\pi i} \int_{\mc C_{1a}(\eta)} g_{\alpha,\beta;N;c_0n^{1/4};\bm\ell}^{[1]}(z) e^{\frac{nz}N} dz\\
= \frac{1}{\sqrt{2}\pi^{N/2} N^N} \cdot \frac{1}{2\pi i} \sum_{r=0}^{R_0-1} \int_{\mc C_{1a}(\eta)} z^{r/2} e^{\frac{\pi^2N}{12z}+\frac{nz}N} \int_{\R^{N-1}} \int_{u_\beta+\kappa\sqrt{z}}^\infty P_r(\bm{u_{[1]}},u_\alpha) du_\alpha d\bm{u_{[1]}} dz\\
+ \sum_{r=0}^{S(R_0)-1} W_r(\kappa) \frac{1}{2\pi i} \int_{\mc C_{1a}(\eta)} e^{\frac{\pi^2N}{12z}+\frac{nz}N-\frac{\kappa^2z}2} z^{(r+1)/2} dz + \frac{1}{2\pi i} \int_{\mc C_{1a}(\eta)} O\Brb{e^{\frac{\pi^2N}{12\varepsilon}}\varepsilon^L} e^{\frac{nz}N} dz.
\end{multline}
We note that $\kappa = \kappa_{\alpha,\beta;N;c_0n^{1/4};\bm\ell}$ can be written explicitly as
\begin{equation}\label{eq:kappa_expansion}
\kappa = c_0 n^{1/4} + \partial^*, \quad \text{ where } \quad  \partial^* = \partial + [\ell_\alpha-\ell_\beta-\lceil c_0 n^{1/4}\rceil]_N,
\end{equation}
where $[\ell]_N$ denotes the smallest non-negative integer congruent to $\ell\pmod{N}$. Since $\mc C_{1a}(\eta)$ has bounded length, the error term in \eqref{eq:arc_1a_decomp} has size $\ll e^{\pi\sqrt{n/3}} n^{-L/2}$. Now we evaluate the terms 
\begin{equation}\label{eq:W_integral_1a}
\frac{1}{2\pi i} \int_{\mc C_{1a}(\eta)} e^{\frac{\pi^2N}{12z}+\frac{nz}N-\frac{\kappa^2z}2} W_r(\kappa) z^{(r+1)/2} dz
\end{equation}
appearing in \eqref{eq:arc_1a_decomp}. To evaluate \eqref{eq:W_integral_1a}, we consider a corresponding integral over $\mc C_{1b}(\eta)$:
\[
\frac{1}{2\pi i} \int_{\mc C_{1b}(\eta)} e^{\frac{\pi^2N}{12z}+\frac{nz}N-\frac{\kappa^2z}2} W_r(\kappa) z^{(r+1)/2} dz.
\]
For $y\gg Y$, we have $e^{\pi^2N/12z} \ll e^{\pi^2N/12\varepsilon} \varepsilon^L$ for every $L\in\N$. Recalling that $W_r$ is a polynomial of degree at most $r$, it follows that
\[
\frac{1}{2\pi i} \int_{\mc C_{1b}(\eta)} e^{\frac{\pi^2N}{12z}+\frac{nz}N-\frac{\kappa^2z}2} W_r(\kappa) z^{(r+1)/2} dz \ll e^{\pi\sqrt{n/3}} n^{-L/2} \quad \fa L\in\N.
\]
So we may extend \eqref{eq:W_integral_1a} to an integral over the major arc, introducing a negligible error:
\begin{multline}\label{eq:W_integral_extend}
\frac{1}{2\pi i} \int_{\mc C_{1a}(\eta)} e^{\frac{\pi^2N}{12z}+\frac{nz}N-\frac{\kappa^2z}2} W_r(\kappa) z^{(r+1)/2} dz\\
= \frac{1}{2\pi i} \int_{\mc C_1(\eta)} e^{\frac{\pi^2N}{12z}+\frac{nz}N-\frac{\kappa^2z}2} W_r(\kappa) z^{(r+1)/2} dz + O\Brb{e^{\pi\sqrt{n/3}} n^{-L/2}}.
\end{multline}

To apply \Cref{lem:NR} to \eqref{eq:W_integral_extend}, we need to shift the path of integration from $\mc C_1(\eta)$ to $\mc C_1(\eta')$, where $\eta' = \frac{\pi N}{2\sqrt{3(n-\kappa^2N/2)}}$. We split the integral in \eqref{eq:W_integral_extend} into three pieces, writing
\begin{multline}\label{eq:W_integral_shift}
\frac{1}{2\pi i} \int_{\mc C_1(\eta)} e^{\frac{\pi^2N}{12z}+\frac{nz}N-\frac{\kappa^2z}2} W_r(\kappa) z^{(r+1)/2} dz = \frac{1}{2\pi i} \int_{\eta(1-\theta i)}^{\eta'(1-\theta i)} e^{\frac{\pi^2N}{12z}+\frac{nz}N-\frac{\kappa^2z}2} W_r(\kappa) z^{(r+1)/2} dz\\
+ \frac{1}{2\pi i} \int_{\mc C_1(\eta')} e^{\frac{\pi^2N}{12z}+\frac{nz}N-\frac{\kappa^2z}2} W_r(\kappa) z^{(r+1)/2} dz + \frac{1}{2\pi i} \int_{\eta'(1+\theta i)}^{\eta(1+\theta i)} e^{\frac{\pi^2N}{12z}+\frac{nz}N-\frac{\kappa^2z}2} W_r(\kappa) z^{(r+1)/2} dz.
\end{multline}
Recall that for $y\gg \varepsilon^{1/2-\delta}$, we have $e^{\pi^2N/12z} \ll e^{\pi^2N/12\varepsilon} \varepsilon^L$ for every $L\in\N$. This implies the first and the third integrals on the right-hand side in \eqref{eq:W_integral_shift} have size $\ll e^{\pi\sqrt{n/3}} n^{-L/2}$ for all $L\in\N$. Meanwhile, using \Cref{lem:NR}, we have for $S\in\N$
\begin{multline}\label{eq:W_asymp1}
\frac{1}{2\pi i} \int_{\mc C_1(\eta')} e^{\frac{\pi^2N}{12z}+\frac{nz}N-\frac{\kappa^2z}2} W_r(\kappa) z^{(r+1)/2} dz = \exp\rb{\pi\sqrt{\frac{n-\kappa^2N/2}3}} \rb{\frac nN-\frac{\kappa^2}2}^{-r/4-1}\\
\times \rb{\sum_{s=0}^{S-1} T_{(r+1)/2,\pi\sqrt{N/12},s} \rb{\frac nN - \frac{\kappa^2}2}^{-s/2} + O\rb{n^{-S/2}}}.
\end{multline}
Using \eqref{eq:kappa_expansion}, we see that the factors in \eqref{eq:W_asymp1} admit Taylor series expansions in $n^{-1/4}$:
\begin{align}
 \exp\rb{\pi\sqrt{\frac{n-\kappa^2N/2}3}} &= \exp\rb{\pi\sqrt{\frac n3} - \frac{c_0^2\pi N}{4\sqrt{3}}} \rb{1+O\rb{n^{-1/4}}},\label{eq:W_Taylor1}\\
 \rb{\frac nN-\frac{\kappa^2}2}^{-r/4-1} &= \rb{\frac nN}^{-r/4-1} \rb{1+O\rb{n^{-1/4}}} \label{eq:W_Taylor2}
\end{align}
Thus we have an asymptotic expansion
\begin{align}\nonumber
&\frac{1}{2\pi i} \int_{\mc C_1(\eta')} e^{\frac{\pi^2N}{12z}+\frac{nz}N-\frac{\kappa^2z}2} W_r(\kappa) z^{(r+1)/2} dz\\
&\hspace{2cm} = e^{\pi\sqrt{n/3}} n^{-(r+3)/4} \rb{\sum_{s=1}^{S-1} A_{r,s,\bm\ell}^{[1]} n^{-s/4} + O\rb{n^{-S/4}}},\label{eq:W_asymp2}
\end{align}
where $A_{r,s,\bm\ell}^{[1]}$ depends only on $\alpha,\beta,c_0,N,\bm\ell,\partial^*$. 

Next we evaluate the integral
\[
\frac{1}{2\pi i} \int_{\mc C_{1a}(\eta)} z^{r/2} e^{\frac{\pi^2N}{12z}+\frac{nz}N} \int_{\R^{N-1}} \int_{u_\beta+\kappa\sqrt{z}}^\infty P_r(\bm{u_{[1]}},u_\alpha)du_\alpha d\bm{u_{[1]}} dz.
\]
We split the inner integral, and write
\begin{multline}\label{eq:main_integral_split}
\frac{1}{2\pi i} \int_{\mc C_{1a}(\eta)} z^{r/2} e^{\frac{\pi^2N}{12z}+\frac{nz}N} \int_{\R^{N-1}} \int_{u_\beta+\kappa\sqrt{z}}^\infty P_r(\bm{u_{[1]}},u_\alpha)du_\alpha d\bm{u_{[1]}} dz\\
= \frac{1}{2\pi i} \int_{\mc C_{1a}(\eta)} z^{r/2} e^{\frac{\pi^2N}{12z}+\frac{nz}N} \int_{\R^{N-1}} \int_{u_\beta+\kappa\sqrt{\eta}}^\infty P_r(\bm{u_{[1]}},u_\alpha)du_\alpha d\bm{u_{[1]}} dz\\ - \frac{1}{2\pi i} \int_{\mc C_{1a}(\eta)} z^{r/2} e^{\frac{\pi^2N}{12z}+\frac{nz}N} \int_\R \int_{u_\beta+\kappa\sqrt{\eta}}^{u_\beta+\kappa\sqrt{z}} P_r(\bm{u_{[1]}},u_\alpha)du_\alpha d\bm{u_{[1]}} dz.
\end{multline}
We may rewrite the first integral in \eqref{eq:main_integral_split} as
\begin{multline*}
\frac{1}{2\pi i} \int_{\mc C_{1a}(\eta)} z^{r/2} e^{\frac{\pi^2N}{12z}+\frac{nz}N} \int_{\R^{N-1}} \int_{u_\beta+\kappa\sqrt{\eta}}^\infty P_r(\bm{u_{[1]}},u_\alpha)du_\alpha d\bm{u_{[1]}} dz\\
= \int_{\R^{N-1}} \int_{u_\beta+\kappa\sqrt{\eta}}^\infty C_r(\bm u) e^{-\bm u^T\bm u} du_\alpha d\bm{u_{[1]}} \cdot \frac{1}{2\pi i} \int_{\mc C_{1a}(\eta)} z^{r/2} e^{\frac{\pi^2N}{12z} + \frac{nz}N} dz. 
\end{multline*}
Using \eqref{eq:kappa_expansion}, the integral over $\bm u$ admits a Taylor series expansion in $n^{-1/4}$. For the integral over $z$, we use the same trick as for \eqref{eq:W_integral_1a} and extend the integral over $\mc C_1(\eta)$ gaining a negligible error, then use \Cref{lem:NR} to obtain an asymptotic expansion. Thus we have for $S\in\N$
\begin{multline}\label{eq:main_asymp}
\frac{1}{\sqrt{2}\pi^{N/2} N^N} \int_{\R^{N-1}} \int_{u_\beta+\kappa\sqrt{\eta}}^\infty C_r(\bm u) e^{-\bm u^T\bm u} du_\alpha d\bm{u_{[1]}} \cdot \frac{1}{2\pi i}\int_{\mc C_{1a}(\eta)} z^{r/2} e^{\frac{\pi^2N}{12z} + \frac{nz}{N}} dz\\
= e^{\pi\sqrt{n/3}} n^{-(r+3)/4} \rb{\sum_{s=0}^{S-1} A_{r,s,\bm\ell}^{[0]} n^{-s/4} + O\rb{n^{-S/4}}},
\end{multline}
where $A_{r,s,\bm\ell}^{[0]}$ depends only on $\alpha,\beta,c_0,N,\bm\ell,\partial^*$. The second integral in \eqref{eq:main_integral_split} reads
\begin{equation}\label{eq:zeta_integral}
\frac{1}{2\pi i} \int_{\mc C_{1a}(\eta)} z^{r/2} e^{\frac{\pi^2N}{12z}+\frac{nz}N} \int_{\R^{N-1}} \int_{u_\beta+\kappa\sqrt{\eta}}^{u_\beta+\kappa\sqrt{z}} P_r(\bm u) du_\alpha d\bm{u_{[1]}} dz.
\end{equation}
Using \Cref{prp:EMF}, we may rewrite the inner integral in \eqref{eq:zeta_integral} as
\begin{multline}\label{eq:zeta_EMF}
\int_{u_\beta+\kappa\sqrt{\eta}}^{u_\beta+\kappa\sqrt{z}} P_r(\bm u) du_\alpha = \frac{\sqrt{z}-\sqrt{\eta}}2 \rb{P_r(\bm{u_{[1]}}, u_\beta+\kappa\sqrt{z}) + P_r(\bm{u_{[1]}}, u_\beta+\kappa\sqrt{\eta})}\\
- \sum_{s=1}^S \frac{B_{2s}(\sqrt{z}-\sqrt{\eta})^{2s}}{(2s)!} \rb{P_r^{(2s-1)}(\bm{u_{[1]}}, u_\beta+\kappa\sqrt{z}) - P_r^{(2s-1)}(\bm{u_{[1]}}, u_\beta+\kappa\sqrt{\eta})}\\
+ O(1)\rb{\sqrt{z}-\sqrt{\eta}}^{2S+1} \int_{u_\beta+\kappa\sqrt{\eta}}^{u_\beta+\kappa\sqrt{z}} \vb{P_r^{(2S+1)}(\bm u)} du_\alpha.
\end{multline}
As $|P_r^{(2S+1)}(\bm u)| \ll e^{-\bm{u_{[1]}}^T\bm{u_{[1]}}}$ and $\sqrt{z}-\sqrt{\eta} \ll \sqrt{\eta} Y \ll \eta^{1-\delta}$, it follows that the contribution of the error term to the integral \eqref{eq:zeta_integral} is bounded by $O(e^{\pi\sqrt{n/3}} n^{-S/2})$. For the remaining terms in \eqref{eq:zeta_EMF}, we replace $P_r^{(j)}(\bm {u_{[1]}},u_\beta+\kappa\sqrt{z})$ by the Taylor series
\begin{multline}\label{eq:zeta_Taylor}
P_r^{(j)}(\bm {u_{[1]}},u_\beta+\kappa\sqrt{z})\\
= \sum_{s=0}^{2S} \frac{(\sqrt{z}-\sqrt{\eta})^s}{s!} P_r^{(j+s)} (\bm {u_{[1]}},u_\beta+\kappa\sqrt{\eta}) + O\rb{\rb{\sqrt{z}-\sqrt{\eta}}^{2S+1} e^{-\bm{u_{[1]}}^T\bm{u_{[1]}}}}. 
\end{multline}
Therefore, with an error of $O(e^{\pi\sqrt{n/3}} n^{-S/2})$, we may replace \eqref{eq:zeta_integral} with a finite sum of integrals of the form
\[
\eta^{k/2} \int_{\R^{N-1}} P_r^{(l)}(\bm{u_{[1]}}, u_\beta+\kappa\sqrt{\eta}) d\bm{u_{[1]}} \cdot \frac 1{2\pi i} \int_{\mc C_{1a}(\eta)} z^{(r+j)/2}  e^{\frac{\pi^2N}{12z} + \frac{nz}{N}} dz
\]
for some $j,k,l \in \N_0$, with $j+k\ge 1$. By extending the integral to $\mc C_1(\eta)$ and applying \Cref{lem:NR}, we obtain an asymptotic expansion of the form
\begin{multline}\label{eq:zeta_asymp}
- \frac{1}{\sqrt{2}\pi^{N/2} N^N} \cdot \frac{1}{2\pi i} \int_{\mc C_{1a}(\eta)} z^{r/2} e^{\frac{\pi^2N}{12z}+\frac{nz}N} \int_{\R^{N-1}} \int_{u_\beta+\kappa\sqrt{\eta}}^{u_\beta+\kappa\sqrt{z}} P_r(\bm u)du_\alpha d\bm{u_{[1]}}dz\\
 = e^{\pi\sqrt{n/3}} n^{-(r+3)/4} \rb{\sum_{s=1}^{S-1} A_{r,s,\bm\ell}^{[2]} n^{-s/4} + O\rb{n^{-S/4}}}.
\end{multline}
Combining the asymptotic expansions \eqref{eq:W_asymp2}, \eqref{eq:main_asymp}, and \eqref{eq:zeta_asymp} yields the following proposition.

\begin{prp}\label{prp:d_asymp}
Let $N\in\N_{\ge 2}$, $1\le\alpha,\beta\le N$, $\alpha\ne\beta$, and $c_0\in\R$ a fixed constant. Then for $R\in\N$, we have as $n\to\infty$
\begin{multline*}
d_{\alpha,\beta;N;c_0n^{1/4}}(n)= e^{\pi\sqrt{n/3}} n^{-3/4} \\
\times \sum_{\substack{\bm\ell\in(\Z/N\Z)^N\\ NH(\bm\ell) \equiv n \ppmod{N}}} \rb{A_{0,0,\bm\ell}^{[0]} + n^{-r/4} \sum_{r=1}^{R-1} \sum_{j+k=r} \rb{A_{j,k,\bm\ell}^{[0]}+A_{j,k,\bm\ell}^{[1]}+A_{j,k,\bm\ell}^{[2]}} + O\rb{n^{-R/4}}},
\end{multline*}
where $A_{j,k,\bm\ell}^{[m]}$ depends only on $\alpha,\beta,c_0,N,\bm\ell,\partial^*$.
\end{prp}

\section{Explicit evaluations and proofs of the theorems}\label{section:coefficients}

For the proof of \Cref{thm:main1}, we need to compute the first two terms of the asymptotic expansion in \Cref{prp:d_asymp}. For this we have to compute $A_{0,0,\bm\ell}^{[0]}, A_{0,1,\bm\ell}^{[0]}, A_{1,0,\bm\ell}^{[0]}, A_{0,1,\bm\ell}^{[1]}$, and $A_{0,1,\bm\ell}^{[2]}$. 

\begin{proof}[Proof of \Cref{thm:main1}]
From \cite[Section 6]{BMRS24AP} we find
\begin{align*}
T_{0,\pi\sqrt{N/12},0} &= \frac{N^{1/4}}{2\sqrt[4]{12}}, & T_{1/2,\pi\sqrt{N/12},0} &= \frac{\sqrt{\pi N}}{4\sqrt{3}}, & C_0(\bm u) &= 1, & C_1(\bm u) &= \sum_{j=1}^N -\frac{j u_j}N + \frac{u_j^3}3.
\end{align*}
We also evaluate the Gaussian integrals
\begin{align*}
\int_{\R^{N-1}} \int_{u_\beta+\kappa\sqrt{\eta}} C_0(\bm u) e^{-\bm u^T\bm u} &= \frac{\pi^{N/2}}2 \erfc\rb{\frac{\kappa\sqrt{\pi N}}{2\sqrt[4]{3n}}},\\
\int_{\R^{N-1}} \int_{u_\beta+\kappa\sqrt{\eta}} C_1(\bm u) e^{-\bm u^T\bm u} &= e^{-\kappa^2\eta/2} \frac{(\beta-\alpha)\pi^{(N-1)/2}}{2\sqrt{2} N}. 
\end{align*}
To compute $A_{0,0,\bm\ell}^{[0]}$ and $A_{0,1,\bm\ell}^{[0]}$, we put these back into \eqref{eq:main_asymp} and obtain
\begin{align*}
&\frac{1}{\sqrt{2}\pi^{N/2}N^N} \int_{\R^{N-1}} \int_{u_\beta+\kappa\sqrt{\eta}}^\infty C_0(\bm u) e^{-\bm u^T\bm u} du_\alpha d\bm{u_{[1]}} \cdot \int_{\mc C_{1a}(\eta)} e^{\frac{\pi^2N}{12z} + \frac{nz}N} dz\\
&\hspace{1cm} = \frac{e^{\pi\sqrt{n/3}}}{2\sqrt{2}N^{N-3/4}n^{3/4}} \erfc\rb{\frac{\kappa\sqrt{\pi N}}{2\sqrt[4]{3n}}} \rb{T_{0,\pi\sqrt{N/12},0} + O\rb{n^{-1/2}}}\\
&\hspace{1cm} = \frac{e^{\pi\sqrt{n/3}}}{8\sqrt[4]{3} N^{N-1} n^{3/4}} \rb{\erfc\rb{\frac{c_0\sqrt{\pi N}}{2\sqrt[4]{3}}} - e^{-c_0^2\pi N/4\sqrt{3}} \frac{\partial^*}{\sqrt[4]{3n}} + O\rb{ n^{-1/2}}}. 
\end{align*}
Thus we have
\begin{align*}
A_{0,0,\bm\ell}^{[0]} &= \frac{1}{8\sqrt[4]{3} N^{N-1}} \erfc\rb{\frac{c_0\sqrt{\pi N}}{2\sqrt[4]{3}}}, & A_{0,1,\bm\ell}^{[0]} &= -e^{-c_0^2\pi N/4\sqrt{3}} \frac{\partial^*}{8\sqrt{3}N^{N-3/2}}.
\end{align*}
Similarly, for $A_{1,0,\bm\ell}^{[0]}$ we compute
\begin{align*}
&\frac{1}{\sqrt{2}\pi^{N/2}N^N} \int_{\R^{N-1}} \int_{u_\beta+\kappa\sqrt{\eta}}^\infty C_1(\bm u) e^{-\bm u^T\bm u} du_\alpha d\bm{u_{[1]}} \cdot \int_{\mc C_{1a}(\eta)} z^{1/2} e^{\frac{\pi^2N}{12z} + \frac{nz}N} dz\\
&\hspace{1cm} = \frac{e^{\pi\sqrt{n/3}}(\beta-\alpha)}{4\sqrt{\pi} N^N n} e^{-\kappa^2\eta/2} \rb{T_{1/2,\pi\sqrt{N/12},0} + O\rb{n^{-1/2}}}\\
&\hspace{1cm} = \frac{e^{\pi\sqrt{n/3}} (\beta-\alpha)}{16\sqrt{3} N^{N-1/2} n} \rb{e^{-c_0^2\pi N/4\sqrt{3}} + O\rb{n^{-1/4}}}. 
\end{align*}
Thus we have
\[
A_{1,0,\bm\ell}^{[0]} = e^{-c_0^2\pi N/4\sqrt{3}} \frac{\beta-\alpha}{16\sqrt{3}N^{N-1/2}}.
\]

For the computation of $A_{0,1,\bm\ell}^{[1]}$, we find from the proof of \Cref{prp:g1_asymp} that
\[
W_0(\kappa) = \frac{1}{\sqrt{2} \pi^{N/2}} V_{0,0}(\kappa) = \frac{1}{\sqrt{2}\pi^{N/2}} \rb{\frac{\pi^{(N-1)/2}}{2\sqrt{2} N^{N-1}}} = \frac{1}{4\sqrt{\pi} N^{N-1}}. 
\]
Putting this back to \eqref{eq:W_asymp1} gives, using \eqref{eq:W_Taylor1} and \eqref{eq:W_Taylor2}
\begin{align*}
&\frac{1}{2\pi i} \int_{\mc C_1(\eta')} e^{\frac{\pi^2N}{12z}+\frac{nz}N-\frac{\kappa^2z}2} W_0(\kappa) z^{1/2}dz\\
&\hspace{1cm} = \frac{1}{4\sqrt{\pi} N^{N-1}} \exp\rb{\pi\sqrt{\frac{n-\kappa^2N/2}3}} \rb{\frac nN-\frac{\kappa^2}2}^{-1} \rb{T_{1/2,\pi\sqrt{N/12},0}+ O\rb{n^{-1/2}}}\\
&\hspace{1cm} = \frac{e^{\pi\sqrt{n/3}}}{16\sqrt{3} N^{N-5/2} n} \rb{e^{-c_0^2\pi N/4\sqrt{3}} + O\rb{n^{-1/4}}}. 
\end{align*}
Thus we have
\[
A_{0,1,\bm\ell}^{[1]} = e^{-c_0^2\pi N/4\sqrt{3}} \frac{1}{16\sqrt{3} N^{N-5/2}}.
\]

Finally we compute $A_{0,1,\bm\ell}^{[2]}$. Using \eqref{eq:zeta_EMF} and \eqref{eq:zeta_Taylor}, we find
\begin{align*}
&-\frac{1}{\sqrt{2}\pi^{N/2} N^N} \cdot \frac{1}{2\pi i} \int_{\mc C_{1a}(\eta)} e^{\frac{\pi^2N}{12z}+\frac{nz}N} \int_{\R^{N-1}} \int_{u_\beta+\kappa\sqrt{\eta}}^{u_\beta+\kappa\sqrt{z}} P_0(\bm u)du_\alpha d\bm{u_{[1]}}dz\\
&\hspace{0.5cm} =-\frac{1}{\sqrt{2}\pi^{N/2} N^N} \cdot \frac{1}{2\pi i} \int_{\mc C_{1a}(\eta)} \rb{\sqrt{z}-\sqrt{\eta}}e^{\frac{\pi^2N}{12z}+\frac{nz}N} dz\\
&\hspace{2cm}\times \int_{\R^{N-1}} e^{-(\bm{u_{[1]}},u_\beta+\kappa\sqrt{\eta})^T(\bm{u_{[1]}},u_\beta+\kappa\sqrt{\eta})} d\bm{u_{[1]}} + O\rb{e^{\pi\sqrt{n/3}} n^{-5/4}}\\
&\hspace{0.5cm} = \frac{e^{\pi\sqrt{n/3}-\kappa^2\eta/2}}{2\sqrt{\pi}N^N} \rb{-\frac Nn T_{1/2,\pi\sqrt{N/12},0} + \frac{N^{3/4}\eta^{1/2}}{n^{3/4}} T_{0,\pi\sqrt{N/12},0} + O\rb{n^{-5/4}}},
\end{align*}
where the first term comes from the $\sqrt{z}$-integral, and the second term comes from the $\sqrt{\eta}$-integral. A direct check shows that these two terms cancel, and we have $A_{0,1,\bm\ell}^{[2]} = 0$. Putting the values of $A_{0,0,\bm\ell}^{[0]}, A_{0,1,\bm\ell}^{[0]}, A_{1,0,\bm\ell}^{[0]}, A_{0,1,\bm\ell}^{[1]}, A_{0,1,\bm\ell}^{[2]}$ back to \Cref{prp:d_asymp} yields \Cref{thm:main1}.
\end{proof}

To prove \Cref{thm:main1a}, we need the following lemma from \cite{BMRS24AP}, which follows by a direct calculation.
\begin{lem}[{\cite[Lemma 6.1]{BMRS24AP}}] \label{lem:l_count}
Let $N\ge 5$, $1\le \alpha,\beta \le N$, $\alpha\ne\beta$, and $r,\ell_\alpha,\ell_\beta\in\Z/N\Z$. Then
\begin{equation}\label{eq:l_count}
\#\cbm{\bm{\ell_{[2]}} \in (\Z/N\Z)^{N-2}}{NH(\bm{\ell_{[2]}}, \ell_\alpha,\ell_\beta) \equiv r\pmod{N}} = N^{N-3},
\end{equation}
where $\bm{\ell_{[2]}} \in (\Z/N\Z)^{N-2}$ runs through the indices $j\ne\cb{\alpha,\beta}$.
\end{lem}

\begin{proof}[Proof of \Cref{thm:main1a}]
The case $N=2$ can be verified directly from \Cref{thm:main1}. Now suppose $N\ge 5$. By \Cref{lem:l_count}, for every $r,\ell_\alpha,\ell_\beta \in \Z/N\Z$, there are precisely $N^{N-3}$ tuples $\bm\ell\in(\Z/N\Z)^N$ with given values of $\ell_\alpha$ and $\ell_\beta$ such that $NH(\bm\ell)\equiv r\pmod{N}$. So we may evaluate the inner sum in \Cref{thm:main1} as follows:
\begin{align*}
&\sum_{\substack{\bm\ell \in (\Z/N\Z)^N\\ NH(\bm\ell) \equiv n\ppmod N}} \rb{\frac{1}{8\sqrt[4]{3}N^{N-1}} \erfc\rb{\frac{c_0\sqrt{\pi N}}{2\sqrt[4]{3}}} + \frac{e^{-c_0^2\pi N/4\sqrt{3}} (N^2-2\partial^*N+(\beta-\alpha))}{16\sqrt{3} N^{N-1/2} n^{1/4}}}\\
&\hspace{0.2cm}= \frac{1}{8\sqrt[4]{3}} \erfc\rb{\frac{c_0\sqrt{\pi N}}{2\sqrt[4]{3}}} + \sum_{\ell_\alpha,\ell_\beta\in \Z/N\Z} \frac{e^{-c_0^2\pi N/4\sqrt{3}} (N^2-2(\partial + [\ell_\alpha-\ell_\beta-\lceil c_0n^{1/4}\rceil]_N)N + (\beta-\alpha))}{16\sqrt{3} N^{5/2} n^{1/4}}\\
&\hspace{0.2cm}= \frac{1}{8\sqrt[4]{3}} \erfc\rb{\frac{c_0\sqrt{\pi N}}{2\sqrt[4]{3}}} + \frac{\beta-\alpha+N-2N\partial}{16\sqrt{3N} n^{1/4}}.
\end{align*}
Putting this back into \eqref{eq:main1} yields \Cref{thm:main1a}.
\end{proof}

When $N=3,4$, for any $r\in \Z/N\Z$ there are still $N^{N-1}$ tuples $\bm\ell\in(\Z/N\Z)^N$ such that $NH(\bm\ell) = r \pmod{N}$. So the main term of the asymptotic formula \eqref{eq:main1} is still independent of the congruence class of $n\pmod{N}$. However, the count \eqref{eq:l_count} no longer holds, so the second term in the asymptotic formula \eqref{eq:main1} can genuinely depend on the congruence classes of $n$ and $\lceil c_0n^{1/4}\rceil \pmod{N}$. 

For a concrete example, consider the case where $N=3$, and $(\alpha,\beta) = (1,2)$. By a direct counting, we get the following asymptotic formula as $n\to\infty$ for $n\equiv r\pmod{N}$ and $\lceil c_0n^{1/4} \rceil \equiv s\pmod{N}$:
\begin{multline*}
d_{1,2;3;c_0n^{1/4}}(n) = e^{\pi\sqrt{n/3}} n^{-3/4}\\
\times \rb{\frac 1{8\sqrt[4]{3}} \erfc\rb{\frac{c_0\sqrt{\pi} \sqrt[4]{3}}2} + \frac{e^{-c_0^2\pi N/4\sqrt{3}} (10-6(\partial +\sigma_{r,s}))}{48 n^{1/4}} + O\rb{n^{-1/2}}},
\end{multline*}
where $\sigma_{r,s}$ is given in the following table:
\[\begin{array}{c|ccc}
\sigma_{r,s} & s=0 & s=1 & s=2\\\hline
r=0 & 0 & 2 & 1\\
r=1 & 1 & 0 & 2\\
r=2 & 2 & 1 & 0
\end{array}\]

Next we prove \Cref{thm:main2,thm:main2a}.

\begin{proof}[Proof of \Cref{thm:main2}] 
Using \Cref{thm:main1} and \eqref{eq:Hua}, we obtain for $c_0\in\R$
\[
\lim_{n\to\infty} \frac{d_{\alpha,\beta;N;c_0n^{1/4}}(n)}{d(n)} = \frac 12 \erfc\rb{\frac{c_0\sqrt{\pi N}}{2\sqrt[4]{3}}}. 
\]
Comparing with the cumulative distribution function for the normal distribution with mean $0$ and variance $\frac{2\sqrt{3}}{\pi N}$ yields \Cref{thm:main2}.
\end{proof}

\begin{proof}[Proof of \Cref{thm:main2a}]
From \cite[Theorem 1.2]{BMRS24AP} we find the asymptotic formula for the parity bias $d_{\alpha,\beta;N;0} - d_{\beta,\alpha;N;0}$:
\begin{equation}\label{eq:bias_asymp}
d_{\alpha,\beta;N;0} - d_{\beta,\alpha;N;0} = e^{\pi\sqrt{n/3}} n^{-1} \rb{\frac{\beta-\alpha}{8\sqrt{3N}} +O\rb{n^{-1/4}}}. 
\end{equation}
Meanwhile, the quantity in question is given by
\[
\lim_{n\to\infty} \frac{\big(d_{\alpha,\beta;N;an^{1/4}}(n) - d_{\alpha,\beta;N;bn^{1/4}}(n)\big) - \big(d_{\beta,\alpha;N;an^{1/4}}(n) - d_{\beta,\alpha;N;bn^{1/4}}(n)\big)}{d_{\alpha,\beta;N;0} - d_{\beta,\alpha;N;0}}.
\]
Applying \Cref{thm:main1a} and the asymptotic formula \eqref{eq:bias_asymp} yields \Cref{thm:main2a}.
\end{proof}

We end the article with some illustrations of the asymptotic behaviour of $d_{\alpha,\beta;N;c_0n^{1/4}}(n)$, as well as the distributions of the parity differences $\pd(\lambda)$ and the parity biases. \Cref{fig:1,fig:2} give the ratio between the asymptotic estimates of $d_{1,2;2;c_0n^{1/4}}(n)$ and $d_{2,1;2;c_0n^{1/4}}(n)$ to their respective true values, for $c_0 = 1$ and $2$. The estimates are coloured according to the following table:
\begin{center}
\begin{tabular}{|l|l|}\hline
$e^{\pi\sqrt{n/3}}n^{-3/4}\frac{1}{8\sqrt[4]{3}}\erfc(\frac{c_0\sqrt{\pi}}{\sqrt[4]{3}})\big/d_{1,2;2;n^{1/4}}(n)$ & Red\\
$e^{\pi\sqrt{n/3}}n^{-3/4}\frac{1}{8\sqrt[4]{3}}\erfc(\frac{c_0\sqrt{\pi}}{\sqrt[4]{3}})\big/d_{2,1;2;n^{1/4}}(n)$ & Orange\\
$e^{\pi\sqrt{n/3}}n^{-3/4}\big(\frac{1}{8\sqrt[4]{3}}\erfc(\frac{c_0\sqrt{\pi}}{\sqrt[4]{3}})+\frac{e^{-c_0^2\pi/2\sqrt{3}}(3-4\partial)}{16\sqrt{6} n^{1/4}}\big)\big/d_{1,2;2;n^{1/4}}(n)$ & Magenta\\
$e^{\pi\sqrt{n/3}}n^{-3/4}\big(\frac{1}{8\sqrt[4]{3}}\erfc(\frac{c_0\sqrt{\pi}}{\sqrt[4]{3}})+\frac{e^{-c_0^2\pi/2\sqrt{3}}(1-4\partial)}{16\sqrt{6} n^{1/4}}\big)\big/d_{2,1;2;n^{1/4}}(n)$ & Blue\\\hline
\end{tabular}
\end{center}

\begin{figure}
\includegraphics[width=11cm]{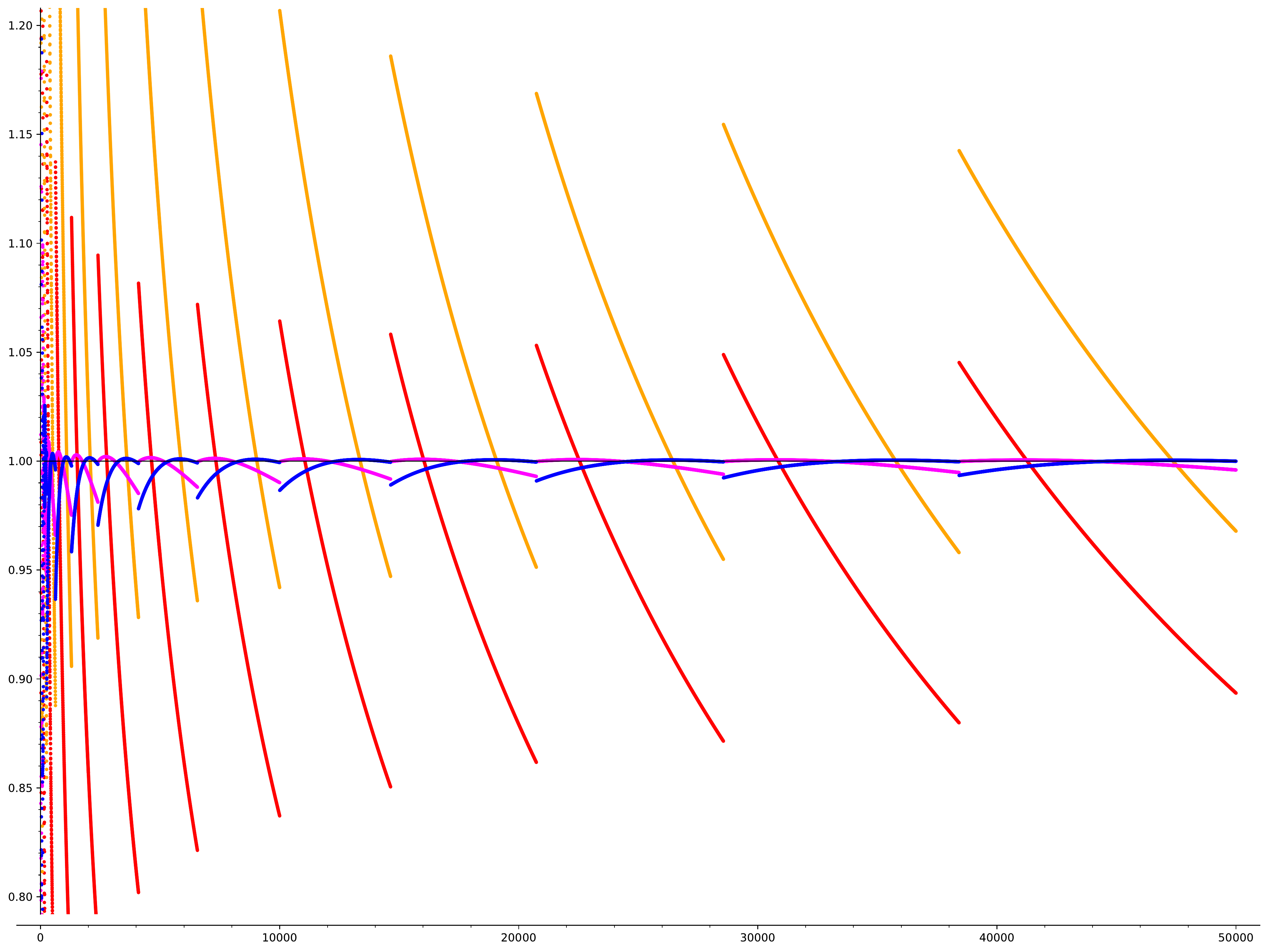}
\caption{The ratios between asymptotic estimates of $d_{1,2;2;c_0n^{1/4}}(n)$ and $d_{2,1;2;c_0n^{1/4}}(n)$ to their true values, $c_0=1$, $1\le n\le 50000$.}\label{fig:1}
\end{figure}

\begin{figure}
\includegraphics[width=11cm]{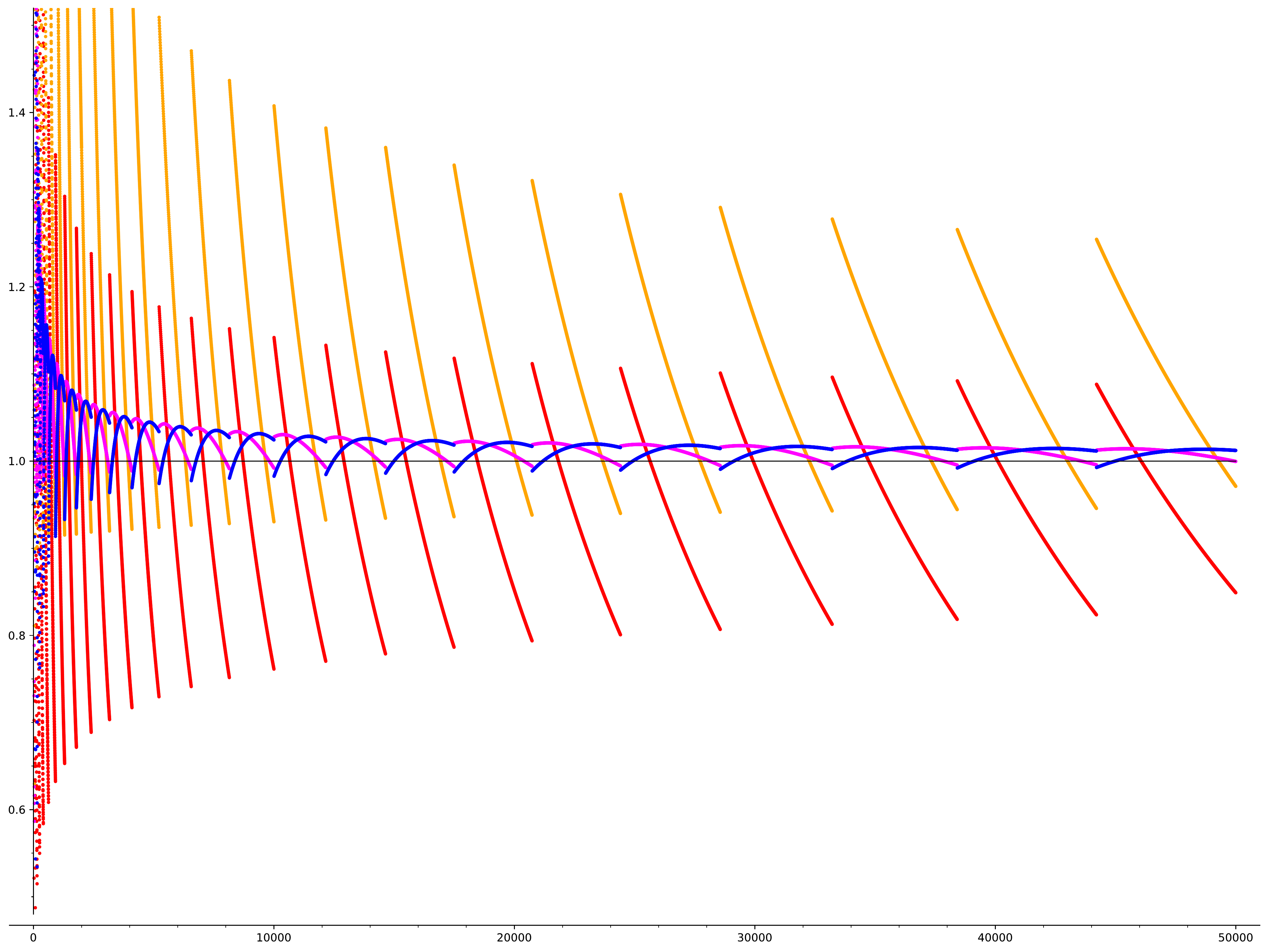}
\caption{The ratios between asymptotic estimates of $d_{1,2;2;c_0n^{1/4}}(n)$ and $d_{2,1;2;c_0n^{1/4}}(n)$ to their true values, $c_0=2$, $1\le n\le 50000$.}\label{fig:2}
\end{figure}

\Cref{fig:3,fig:4} illustrate the parity differences and parity biases of partitions of $n=50000$ into distinct parts, with the height normalised to $1$. The estimates given in \Cref{thm:main2,thm:main2a} are shown in red.

\begin{figure}
\includegraphics[width=11cm]{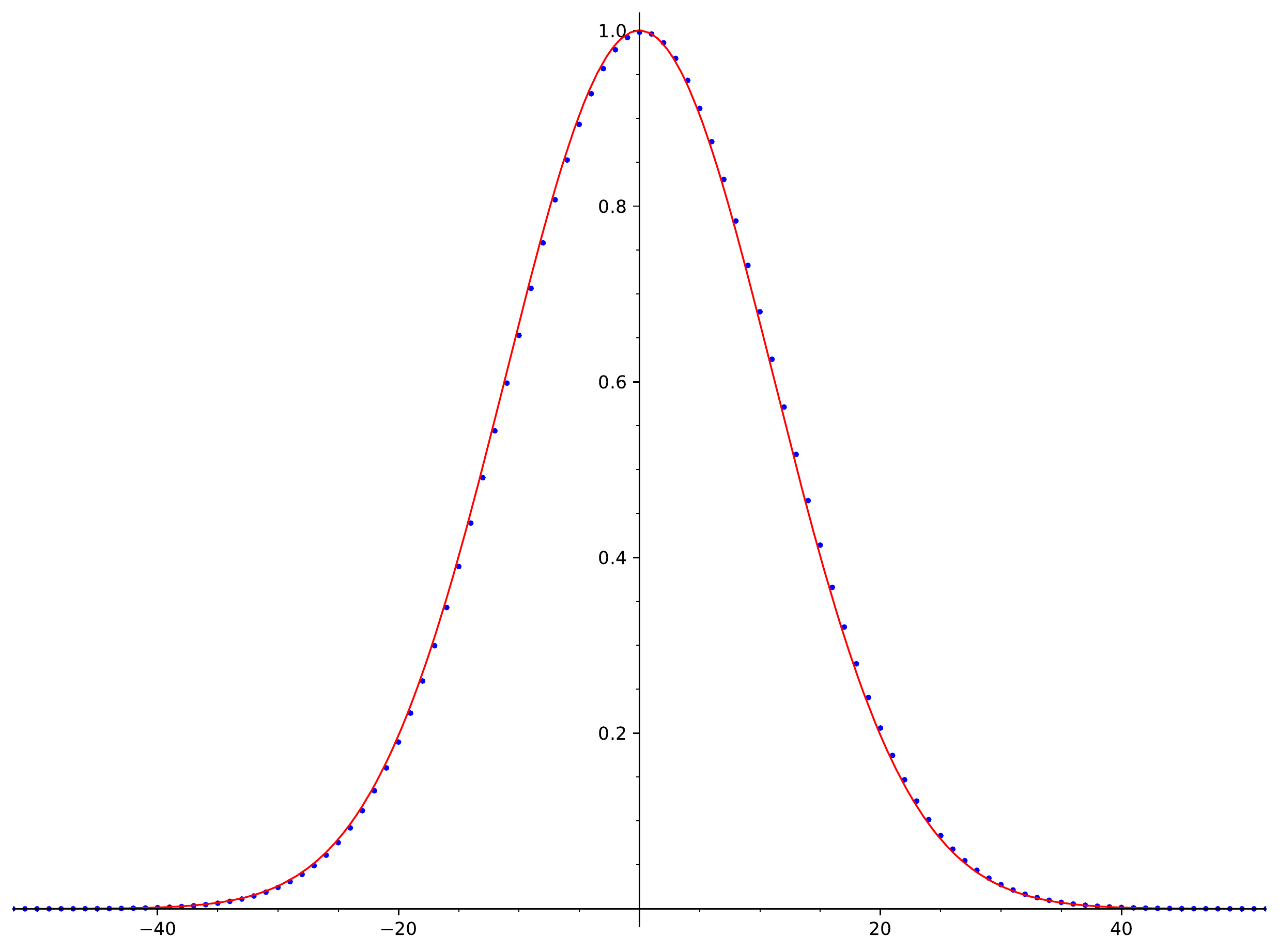}
\caption{$\#\cbm{\lambda \in \mc D(50000)}{\pd_{1,2;2}(\lambda)=k}$ for $-50\le k \le 50$, normalised.}\label{fig:3}
\end{figure}

\begin{figure}
\includegraphics[width=11cm]{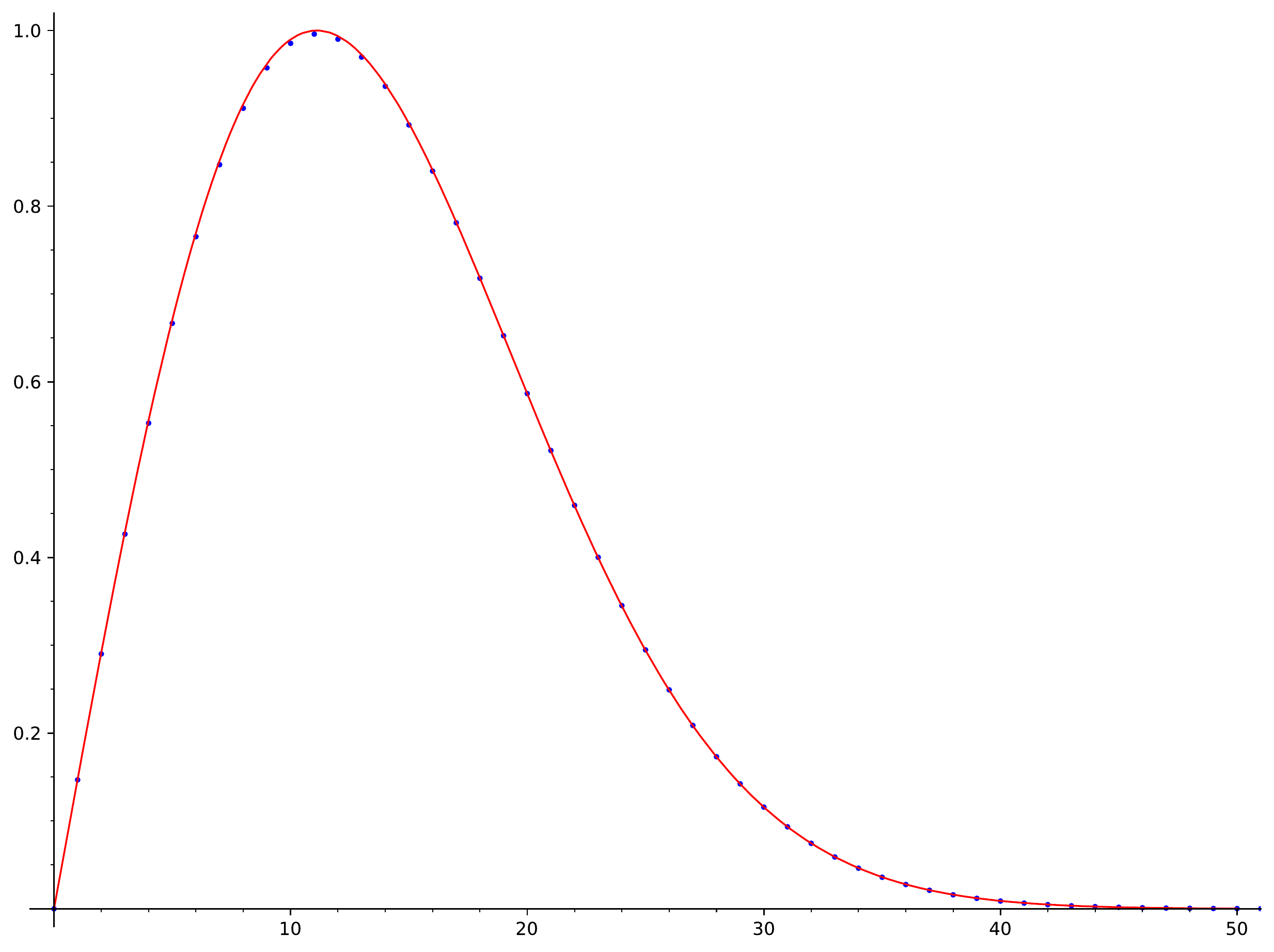}
\caption{$\#\{\lambda \in \mc D(50000)\;|\;\pd_{1,2;2}(\lambda)=k\} - \#\{\lambda \in \mc D(50000)\;|\:\pd_{2,1;2}(\lambda)=k\}$ for $0\le k \le 50$, normalised.}\label{fig:4}
\end{figure}

\newcommand{\etalchar}[1]{$^{#1}$}


\begin{thebibliography}{BBD{\etalchar{+}}22}

\bibitem[BBD{\etalchar{+}}22]{BBDMS22PB}
K.~Banerjee, S.~Bhattacharjee, M.~G. Dastidar, P.~J. Mahanta, and M.~P. Saikia.
\newblock Parity biases in partitions and restricted partitions.
\newblock {\em European J. Combin.}, 103:103522, 2022.

\bibitem[BMRS24]{BMRS24AP}
K.~Bringmann, S.~H. Man, L.~Rolen, and M.~Storzer.
\newblock Asymptotics of parity biases for partitions into distinct parts via
  {N}ahm sums.
\newblock {\em Proc. Lond. Math. Soc. (3)}, 129(6):e70010, 2024.

\bibitem[Che22]{Che22FR}
S.~Chern.
\newblock Further results on biases in integer partitions.
\newblock {\em Bull. Korean Math. Soc.}, 59(1):111--117, 2022.

\bibitem[DS07]{DS07DR}
C.~Dartyge and M.~Szalay.
\newblock Dominant residue classes concerning the summands of partitions.
\newblock {\em Funct. Approx. Comment. Math.}, 37:65--96, 2007.

\bibitem[DS11]{DS11LD}
C.~Dartyge and M.~Szalay.
\newblock Local distribution of the parts of unequal partitions in arithmetic
  progressions {I}.
\newblock {\em Publ. Math. Debrecen}, 79(3-4):379--393, 2011.

\bibitem[DS12]{DS12LD}
C.~Dartyge and M.~Szalay.
\newblock Local distribution of the parts of unequal partitions in arithmetic
  progressions {II}.
\newblock {\em Publ. Math. Debrecen}, 81(3-4):453--486, 2012.

\bibitem[DSS06]{DSS06DS}
C.~Dartyge, A.~S{\'a}rk{\"o}zy, and M.~Szalay.
\newblock On the distribution of the summands of unequal partitions in residue
  classes.
\newblock {\em Acta Math. Hungar.}, 110(4):323--335, 2006.

\bibitem[GZ21]{GZ21AN}
S.~Garoufalidis and D.~Zagier.
\newblock Asymptotics of {N}ahm sums at roots of unity.
\newblock {\em Ramanujan J.}, 55:219--238, 2021.

\bibitem[Hua42]{Hua42NP}
L.-K. Hua.
\newblock On the number of partitions of a number into unequal parts.
\newblock {\em Trans. Amer. Math. Soc.}, 51(1):194--201, 1942.

\bibitem[KK21]{KK21BI}
B.~Kim and E.~Kim.
\newblock Biases in integer partitions.
\newblock {\em Bull. Aust. Math. Soc.}, 104(2):177--186, 2021.

\bibitem[KKL20]{KKL20PB}
B.~Kim, E.~Kim, and J.~Lovejoy.
\newblock Parity bias in partitions.
\newblock {\em European J. Combin.}, 89:103159, 2020.

\bibitem[NR17]{NR17IP}
H.~Ngo and R.~Rhoades.
\newblock Integer partitions, probabilities and quantum modular forms.
\newblock {\em Res. Math. Sci.}, 4:17, 2017.

\bibitem[VZ11]{VZ11NC}
M.~Vlasenko and S.~Zwegers.
\newblock Nahm's conjecture: Asymptotic computations and counterexamples.
\newblock {\em Commun. Number Theory Phys.}, 5(3):617--642, 2011.

\bibitem[Zag06]{Zag06MT}
D.~Zagier.
\newblock The {M}ellin transform and related analytic techniques.
\newblock In {\em Appendix to E. Zeidler, Quantum Field Theory I: Basics in
  Mathematics and Physics}, chapter 6.7, pages 305--323. Springer Berlin
  Heidelberg, 2006.

\bibitem[Zag07]{Zag07DF}
D.~Zagier.
\newblock The dilogarithm function.
\newblock In P.~Cartier, P.~Moussa, B.~Julia, and P.~Vanhove, editors, {\em
  Frontiers in Number Theory, Physics, and Geometry II}, pages 3--65. Springer
  Berlin-Heidelberg, 2007.

\end{thebibliography}
\end{document}